\documentclass[9pt]{amsart}
\usepackage{amssymb,amsthm,amsmath}
\usepackage[numbers,sort&compress]{natbib}
\usepackage{color}
\usepackage{graphicx}
\usepackage{tikz}
\usepackage{arydshln}
\title[ ]{on the Kolmogorov theorem for some infinite-dimensional Hamiltonian systems of short range}
\author{ Yuan wu and Xiaoping Yuan}

\address{School of Mathematical Sciences, Fudan University, Shanghai 200433, P. R. China} \email{14110840003@fudan.edu.cn}
\address{School of Mathematical Sciences, Fudan University, Shanghai 200433, P. R. China}
\email{xpyuan@fudan.edu.cn}


\theoremstyle{plain}
\newtheorem{theorem}{Theorem}[section]
\newtheorem{corollary}[theorem]{Corollary}
\newtheorem{lemma}[theorem]{Lemma}

\theoremstyle{definition}

\newtheorem{remark}[theorem]{Remark}

 \numberwithin{equation}{section}

\begin{document}


\begin{abstract}
In this paper,\ it is proved that the infinite KAM torus with prescribed frequency exists in a sufficiently small neighborhood of a given $ I^{0}$ for nearly integrable and analytic Hamiltonian system $  H(I,\theta) = H_{0}(I)+ \epsilon H_{1}(I,\theta)$ of infinite degree of freedom and of short range.\ That is to say, we will give an extension of the original Kolmogorov theorem to the infinite-dimensional case of short range.\ The proof is based on the approximation of finite-dimensional Kolmogorov theorem and an improved KAM machinery which works for the normal form depending on initial $ I^{0}$.
\end{abstract}
\maketitle
\section{introduction and main results}
\subsection{Motivations}
Since Kolmogorov's work \cite{Kol1954On} in 1954,\ remarkable results have been obtained in perturbation theory of integrable Hamiltonian systems.\ Here we share some of the most important extensions based on the original Kolmogorov theorem for a better understanding of this perturbation theory.

We recall the fact that Kolmogorov \cite{Kol1954On} announced that most invariant tori for the integrable Hamiltonian systems persist under small perturbations.\ More concretely,\ consider the nearly integrable Hamiltonian system of $n$-freedom
\begin{eqnarray}
\label{e0} H(I,\theta) = H_{0}(I)+ \epsilon H_{1}(I,\theta),
\end{eqnarray}
with the symplectic structure $ dI\wedge d\theta$ on $ \mathbb{R}^{n}\times\mathbb{T}^{n}$ and the action-angle variables $ (I,\theta)$ belonging to some domain $ \mathcal{D}\times\mathbb{T}^{n}\subset \mathbb{R}^{n}\times\mathbb{T}^{n}$.\ If $ H_{0}(I) $ is analytic and satisfies the non-degenerate condition
\begin{eqnarray}
\label{e1}det(\partial^{2}H_{0}(I))\neq 0,\  I\in \mathcal{D},
\end{eqnarray}
Kolmogorov theorem claims that any invariant tori of the unperturbed $ H_{0}$ with prescribed Diophantine frequency $ \omega=\omega(I^{0})= \frac{\partial H_{0}(I)}{\partial I}|_{I=I^{0}}$ for some $ I^{0} \in \mathcal{D}$ persist under a small analytic perturbation $ \epsilon H_{1}(I,\theta)$.\ Actually,\ Kolmogorov also gave a precise outline of its proof which is based on a fast convergent Newton scheme.\ Roughly speaking,\ one can set up a Newton scheme by replacing $ H_{j}$ with $ H_{j+1} = H_{j} \circ \Phi_{j} $ after taking symplectic transformation $ \Phi_{j}$ such that the new Hamiltonian $ H_{j+1} $ is of super-exponential decay.\ Taking $j=1$ as an example,\ the map $\Phi_{1}$ is close to the identity and can be regarded as the composition of two ``elementary'' symplectic transformations: $\Phi_{1}= \Phi^{(1)}_{1}\circ\Phi^{(2)}_{1}$,\ where $ \Phi^{(2)}_{1}: (I',\theta')\rightarrow (\eta,\xi)$ is the symplectic map generated by  $ I'\cdot \xi + \epsilon I'\cdot a(\xi)$,\ while $ \Phi^{(1)}_{1}: (\eta,\xi)\rightarrow(I,\theta)$ is the angle-dependent translation generated by  $ \eta\cdot \theta + \epsilon (b\cdot\theta + s(\theta))$ with real-analytic functions $ a(\xi)$ and $ s(\theta)$ being of zero average on $\mathbb{T}^{n}$ and $ b \in \mathbb{R}^{n}$.\ Obviously,\ $ \Phi^{(2)}_{1} $ acts in ``angle direction'' and will be needed to straighten out the flow up to order $ \mathcal{O}(\epsilon^{2})$,\ while $ \Phi^{(1)}_{1}$ acts in ``action direction'' and will be needed to keep the frequency $ \omega(I^{0})$ of the torus fixed.\ Indeed,\ the vector $ b$ above and the non-degenerate condition (\ref{e1}) are sufficient to overcome the frequency drift(which is a key idea introduced by Kolmogorov).\ See
\cite{chierchia2009kol},\ for the details.\ Later,\ the rigorous proof was given by Arnold\ \cite{arnol1963proof} in the analytic category,\ while Moser\ \cite{Moser1962On} also proved it for the finitely differentiable exact symplectic mappings.\ This theorem now is called the classical KAM theorem.\ See \cite{broer2009quasi},\ and the references therein.

Naturally,\ it is hoped that the KAM theorem of finite-dimensional Hamiltonian systems can be extended to infinite-dimension ones.\ Different from the finite-dimensional case,\ the KAM theorem of infinite-dimensional Hamiltonian systems is generally wrong if it is assumed only that the perturbation is small and sufficiently smooth.\ Hence,\ two slightly special cases,\ Hamiltonian partial differential equations (PDEs) and the Hamiltonian systems defined on the infinite lattices, have been taken into account.\ On the one hand,\ the infinite-dimensional KAM theory has seen enormous progress with application to Hamiltonian PDEs since Kuksin \cite{kuksin1987ha} and Wayne \cite{Wayne1990Pe}.\ As an example to which infinite-dimensional KAM theory applies,\ consider the nonlinear Schr\"{o}dinger equation(NLS)
\begin{eqnarray}\label{1}
\sqrt{-1}u_{t}-\triangle u + V(x,\omega)u+ |u|^{2}u + {h}.{o}.{t}=0,
\end{eqnarray}
subject to Dirichlet condition.\ Kuksin \cite{Kuksin1993Nearly} showed that (\ref{1}) possesses lower dimensional invariant tori around $ u=0$ for ``most'' parameters $ \omega$.\ See \cite{be2018kam,bo2005green,chier2000kam,craig1993ne,
eliasson2010kam,el2016kam,fe2015quasi,liu2011kam,liu2010spe,kap2003kam,
ku1988pe,ku1996Inv,procesi2015kam,
P1996A,xue2018kam,yuan2018kam,zhang2011kam,Wayne1990Pe},\ for more related results.\ In those literatures,\ the frequency vectors $ \omega\in\mathbb{R}^{n} $ or initial data are regarded as the parameters.\
In \cite{El1988Pe},\ the frequency
\[
\omega = \omega_{0}t,
\]
where $  \omega_{0}\in\mathbb{R}^{n}$ is a fixed Diophantine vector and $ t\in\mathbb{R}$ is a parameter.\ Even for the finite-dimensional Hamiltonian,\ Bourgain \cite{bo1997mel} showed that,\ at least,\ a 1-dimensional parameter is needed to guarantee the existence of the  KAM torus with a fixed frequency,\ which is a progress of the results in \cite{El1988Pe,Melnikov1965ON}.\ Since the freedom of Hamiltonian PDEs is infinite and the dimension of the obtained KAM tori is finite,\ the KAM tori are called lower-dimensional tori.\
Note that the dimension of the invariant tori equals to the freedom of the Hamiltonian in the  Kolmogorov theorem.\ Thus,\ those lower-dimensional tori for Hamiltonian PDEs are not in the class of  Kolmogorov's tori in the rigorous sense.\ In order to make sure that the invariant tori are in the class of Kolmogorov's,\ the dimension of the obtained tori should be infinite.\ In this direction,\ It is proved in \cite{bo2005on,cong2018stability} that 1-dimensional NLS has a full dimensional KAM torus of prescribed frequencies with the actions of the tori decaying hyper-exponentially to zero by replacing the potential $ V(x,\omega)$ as infinite-dimensional parameter vectors or replacing  $ V(x,\omega) $ by random Fourier multiplier $ M $ ($ M$ is defined by $ \widehat{Mu}(n)= V_{n}\hat{u}(n)$ for $ n\in \mathbb{Z}$).\ Considering that the Hamiltonian does not involve any exterior parameters in the Kolmogorov theorem,\ the full-dimensional invariant tori in \cite{bo2005on,cong2018stability} are not yet in the class of Kolmogorov's.

On the other hand,\ let us consider models in mathematical physical which consist of lattices of harmonic oscillators with independent identically distributed random frequencies,\ subject to un-harmonic coupling force which are of finite range,\ short range,\ or hierarchy.\ Bellissard,\ Vittot\  \cite{vittot1985in},\ and
Fr\"{o}hlich,\ Spencer,\ Wayne \cite{J1986Local} showed that there is a set $ \Omega \subseteq \mathbb{R}^{\infty}_{+}$ with $Prob(\Omega) > 0 $ (where ``$Prob$" is some probability measure) such that for some $ \omega = (\omega_{i})_{i\in \mathbb{Z}} \in \Omega $,\ there exists an infinitely dimensional KAM-torus for the following Hamiltonian
\begin{eqnarray}\label{e0*}
H = H(I,\theta) = \underset{j\in\mathbb{Z}^{d}}{\sum}\omega_{j}I_{j} + \epsilon P(I,\theta),
\end{eqnarray}
where $ 0< \epsilon\ll 1, d\geq 1$  and $ P $ is of short range.\ Afterwards,\ different kinds of systems with short range have been deeply investigated by many authors.\ See \cite{CY2007A,P1990Small,Yuan2002Con,yuan2014kam},\ for example.\ It is easily observed that such infinite-dimensional Hamiltonian systems are well approximated by finite-dimensional ones and consequently the classical,\ finite-dimensional KAM technique works in this case.\ That is,\ these infinite invariant tori are obtained by successive small perturbations of finite-dimensional tori.\ Such results were obtained in \cite{chierchia1994max,P1990Small,P2002On} by a somewhat similar method.\ However,\ the frequencies of those full-dimensional KAM tori are not prescribed,\ so they are not in the class of Kolmogorov's.

In the present paper,\ we will construct a Kolmogorov theorem for the following infinite-dimensional  Hamiltonian system
\begin{eqnarray}\label{0**}
  H(I,\theta) = H_{0}(I)+ \epsilon H_{1}(I,\theta), (I,\theta) \in \mathbb{C}^{\mathbb{Z}}\times\mathbb{T}^{\mathbb{Z}} = \mathbb{C}^{\mathbb{Z}}\times (\mathbb{C}/2\pi{\mathbb{Z}})^{\mathbb{Z}}
\end{eqnarray}
provided that $H_{0}(I),H_{1}(I,\theta)$ are analytic and of short range.\ Thus,\ in order to construct a KAM torus of prescribed frequency,\ the key difficulty is to eliminate the drift of the frequency.
We will overcome the difficulty by advantage of the facts that the Hamilton system (\ref{0**}) is of short range and the perturbation is high orders of action variables $(I_{j})_{j\in \mathbb{Z}}$ taking an exponential norm with weight $ e^{|j|^{1+\alpha}}$ for any $ \alpha > 0$.\ In other words,\ the strength of the action variable $ I_{j}$ decays so fast that we can regard an infinite-dimensional Hamiltonian as a finite- dimensional one with the help of short range in each step of KAM iteration.\ For a finite-dimensional Hamiltonian,\ the Kolmogorov's original idea works well in $ k$-th KAM iteration by replacing the symplectic $ \Phi_{k}$ with a time-1-map $ X^{1}_{\tilde{F}_{k}}$ of a Hamiltonian vector field $X_{\tilde{F}_{k}}$,\ where $ \tilde{F}_{k}$ is of the form $ \tilde{F}_{k}={F}_{k}+ \langle b^{k},\theta\rangle$ with a Hamiltonian $ {F}_{k}$ to be determined.\ More exactly,\ one can choose a length scale ($\mathbb{L}^{k}$) such that we only need to consider the finite Hamiltonian $ H(\theta,I(k))$ in $ k$-th KAM iteration,\ where $ I(k) = (I_{j})_{|j|\leq \mathbb{L}^{k}-1}$.\ The trouble is estimating vector $ b^{k}$ whose dimensional number $ 2\mathbb{L}^{k}+1$ tends to $ \infty$ as $ k$ to $\infty$.\ Indeed,\ our work is mainly based on the combination of Kolmogorov's original idea \cite{Kol1954On} and that of Fr\"{o}hlich, Spencer, Wayne \cite{J1986Local}
and P\"{o}schel \cite{P1990Small}.\ The main aim of the present paper is to prove that there exist full dimensional KAM tori with the prescribed Diophantine frequency $ \omega= \omega(I^{0})= \frac{\partial H_{0}(I^{0}) }{\partial I} $ for the Hamiltonian (\ref{0**}) in the absence of exterior parameters.

\subsection{The main results}
To state our results we firstly introduce some notations.\ In this paper,\ the positive constant $\alpha$ is fixed. Note
\begin{eqnarray}
\nonumber \mathbb{C}^{\mathbb{Z}}:= \{I = (I_{j})_{j\in \mathbb{Z}}:I_{j}\in \mathbb{C}\},
\end{eqnarray}
where the norms on $ \mathbb{C}^{\mathbb{Z}}$ are defined by
\[
\|I\| := \sum_{j\in \mathbb{Z}}|I_{j}|\exp(|j|^{1+\alpha}),
|I|_{\infty}:= \sup_{j\in \mathbb{Z}}|I_{j}|.
\]

Fix $ I^{0}\in \mathbb{C}^{\mathbb{Z}}$ with $ 0 <\parallel I^{0} \parallel < 1 $.\ Given some $ s > 0 $ and $ r > 0$,\ we define domain
\[
\widetilde{\mathcal{D}}_{s,r}:=\{(I,\theta)\in \mathbb{C}^{\mathbb{Z}} \times \mathbb{T}^{\mathbb{Z}}: \| I-I^{0}\| < s, |Im \theta|_{\infty} < r\},
\]
where $ \mathbb{T}^{\mathbb{Z}}= (\mathbb{C}/2\pi{\mathbb{Z}})^{\mathbb{Z}}$,\ and a phase space
\[
\nonumber \mathcal{P}: \mathbb{C}^{\mathbb{Z}} \times \mathbb{T}^{\mathbb{Z}},
\]
with
\begin{eqnarray}
\nonumber |(I,\theta)|_{\mathcal{P}}= \max(\|I\|,|\theta|_{\infty}).
\end{eqnarray}
Given a sequence of length scales $ (\mathbb{L}^{k}) $.\ Let us consider a vector $ I(k) = (I_{j})_{|j|\leq \mathbb{L}^{k}-1}$ and a matrix $ B = (B_{ij})_{|i|,|j|\leq \mathbb{L}^{k}-1} $.\ We can expand $ I(k) $ into
\[
\widetilde{I} = (\widetilde{I}_{j} : j\in \mathbb{Z}),\ \mbox{here}\ \widetilde{I}_{j}= \begin{cases}
 {I}_{j}, |j|\leq \mathbb{L}^{k}-1\\
 0, |j|\geq \mathbb{L}^{k},
 \end{cases}
\]
and also expand $ B$ into
\[
\bar{B} = (\bar{B}_{ij} : i,j\in \mathbb{Z}),\ \mbox{here}\ \bar{B}_{ij}= \begin{cases}
 {B}_{ij}, |i|,|j|\leq \mathbb{L}^{k}-1\\
 0, |i|\  \mbox{or}\  |j|\geq \mathbb{L}^{k},
 \end{cases}
\]
Define $ \|I(k)\|= \|\widetilde{I}\| $.\ Similarly,\ define $ |||B||| = |||\bar{B}|||$,\ where $ |||\cdot|||$ is the operator norm reduced by $ \|\cdot\|$
from
$ \mathbb{C}^{\mathbb{Z}}$ to $\mathbb{C}^{\mathbb{Z}}$.\ That is,\ we can define
\[
|||B|||= ||| \bar{B}|||= \sup_{I,\|I\|\neq0}\frac{\|\bar{B} I\|}{\|I\|}.
\]

For a map $ G : \widetilde{\mathcal{D}}_{s,r}\rightarrow \mathbb{C}$,\ we define
\[
| G |_{s,r} := \sup_{(I,\theta)\in \widetilde{\mathcal{D}}_{s,r}} | G(I,\theta) |.
\]

Consider an infinite-dimensional Hamiltonian system
\numberwithin{equation}{section}
\begin{eqnarray}\label{maineq}
H(I,\theta)= H_{0}(I) + \epsilon H_{1}(I,\theta),\   (I,\theta)\in\widetilde{\mathcal{D}}_{s,r}
\end{eqnarray}
with the standard symplectic structure  $ d\theta \wedge dI $ on $ \mathbb{C}^{\mathbb{Z}}\times \mathbb{T}^{\mathbb{Z}}$.\ Assume the unperturbed Hamiltonian
\[
H_{0}(I)= \underset{\lfloor i,j\rceil}{\sum}h_{\lfloor i,j\rceil}(I_{i},I_{j}) = \underset{j\in \mathbb{Z}}{\sum}\underset{|i-j|\leq 1}{\sum}h_{( i,j)}(I)
\]
satisfies the following conditions:\\

$(\mathcal{A}_0)$ For any $ j\in \mathbb{Z} $,\ the function $ h_{\lfloor i,j\rceil} :  \tilde{\mathcal{D}}_{s,r}\rightarrow \mathbb{C}$ is real for real arguments,\ analytic in  variables $ (I,\theta) $;

$(\mathcal{A}_1)$ Given two sequences of positive numbers $ (\mathbb{L}^{k}) $ and $ (M^{k}) $.\ Denote $ \omega =(\omega_{j})_{j\in \mathbb{Z}}=(\omega_{j}(I^{0}))_{j\in \mathbb{Z}}= \frac{\partial H_{0}(I^{0}) }{\partial I}$.\ For all $ |j|\leq \mathbb{L}^{k}-1
$,\ denote $ \omega(k) = (\omega_{j})_{|j|\leq \mathbb{L}^{k}-1}$.\ Frequency $\omega $ satisfies the Diophantine conditions
\begin{eqnarray}
|\langle\omega(k),\nu\rangle| \geq (\epsilon^{(1+\beta)^{k-1}})^{\gamma}
\end{eqnarray}
for $ 0 < |v| < M^{k}, k\geq 1 $ and $ \beta> 0,\gamma> 0$;

$(\mathcal{A}_2)$ Given a sequence of positive numbers $ (\mathbb{L}^{k})$.\ Denote $ \Omega=\frac{\partial^{2}H_{0}(I^{0})}{\partial I^{2}} = (\Omega_{ij})_{i,j\in \mathbb{Z}}$.\
For any $ i,j\in \mathbb{Z}$,\ note $ \Omega(k) = (\Omega_{ij})_{|i|,|j|\leq \mathbb{L}^{k}-1}$.\ The operator $ \bar{\Omega}(k): \mathbb{C}^{\mathbb{Z}}\rightarrow\mathbb{C}^{\mathbb{Z}}$  satisfies
\[
|||\bar{\Omega}(k)|||\leq \kappa_{1},
\]
and its inverse operator $\bar{\Omega}^{-1}(k): \mathbb{C}^{\mathbb{Z}}\rightarrow\mathbb{C}^{\mathbb{Z}}$ satisfies
\[
|||\bar{\Omega}^{-1}(k)|||\leq \kappa_{2},
\]
for any $k\geq 1 $ and $ 0< \kappa_{1},\kappa_{2}< \infty$.\\
Also assume the perturbed Hamiltonian
 \[
 H_{1}= \underset{\lfloor i,j\rceil}{\sum}f_{\lfloor i,j\rceil}(I_{i},I_{j},\theta_{i},\theta_{j}) = \underset{j\in \mathbb{Z}}{\sum}\underset{|i-j|\leq 1}{\sum}f_{( i,j)}(I,\theta)
 \]
satisfies the following conditions:

$(\mathcal{B}_0)$ For any $ j\in \mathbb{Z} $,\ the function $ f_{\lfloor i,j\rceil} :  \tilde{\mathcal{D}}_{s,r}\rightarrow \mathbb{C}$ is real for real arguments,\ analytic in  variables $ (I,\theta) $;

$(\mathcal{B}_1)$ For any $ j\in \mathbb{Z} $,\ we have
\[
|f_{\lfloor i,j\rceil}(I,\theta)|\sim \mathcal{O}(|I_{i}|^{\iota_{i}}|I_{j}|^{\iota_{j}}),\ \mbox{for all}\  \lfloor i,j\rceil\  \mbox{and}\  \iota_{i}+\iota_{j}\geq 5,
\]
where $ \mathcal{O}(x)$ means higher order infinitesimal of $x$.\\

Our main result is as follows

\begin{theorem}\label{0*}
Suppose $ H = H_{0} + \epsilon H_{1}$ defined in (\ref{maineq}) satisfies $ (\mathcal{A}_0)-(\mathcal{A}_2)$ and $ (\mathcal{B}_0)-(\mathcal{B}_1)$.\ Then for any
$ 0 < \gamma< \frac{1}{301}$ and $ I^{0}\in\mathbb{C}^{\mathbb{Z}}$ with $ 0 < \|I^{0}\|< 1 $,\ there exists a constant $ \epsilon_{0}= \epsilon_{0}(s,r,I^{0},\gamma)> 0$ such that,\ for $ 0 < \epsilon \leq \epsilon_{0}$,\ there is a set $ \mathcal{R}^{\infty}(I^{0}) $,\ a measure $ \mu $ with $
\mu(\mathcal{R}^{\infty}) = 1-\sum^{\infty}_{j=0}\epsilon_{j}^{\kappa}$ ($ 0 < \kappa < \gamma $) and a real-analytic symplectic transformation
$\Psi_{*}:   \tilde{\mathcal{D}}_{s,r}\rightarrow \tilde{\mathcal{D}}_{\frac{1}{2}s,\frac{1}{2}r}$,\ such that for each $ \omega(I^{0}) $ in $ \mathcal{R}^{\infty} $
\begin{eqnarray}
 H\circ\Psi_{*} = e_{*}+ \langle\omega(I^{0}),I-I^{0}\rangle +\frac{1}{2}\langle\Omega_{*}(I-I^{0}),I-I^{0}\rangle + P_{*}(I-I^{0},\theta),
\end{eqnarray}
where $ |P_{*}| = \mathcal{O}(|I_{i}-I^{0}_{i}|^{\iota_{i}}|I_{j}-I^{0}_{j}|^{\iota_{j}}|I_{k}-I^{0}_{k}|^{\iota_{k}})$ for $ {\iota_{i}}+{\iota_{j}}+{\iota_{k}}\geq 3$.\\
Furthermore,\ one has
\begin{eqnarray}
|\Psi_{*}-id|_{\tilde{\mathcal{D}}_{\frac{1}{2}s,\frac{1}{2}r}}\lessdot \epsilon^{\frac{17}{50}},| e_{*}-e|\lessdot 2\epsilon,
\end{eqnarray}
and the operator $ \Omega_{*}: \mathbb{C}^{\mathbb{Z}}\rightarrow \mathbb{C}^{\mathbb{Z}} $ satisfies
\begin{eqnarray}
\nonumber|||\Omega_{*}-\Omega||| \lessdot \epsilon^{\frac{1}{5}}.
\end{eqnarray}
\end{theorem}

From Theorem \ref{0*},\ the following corollary can be obtained.

\begin{corollary}\label{1*}
Consider the Hamiltonian
\begin{eqnarray}\nonumber
H_*=H\circ\Psi_{*}= e_{*}+ \langle\omega(I^{0}),I-I^{0}\rangle +\frac{1}{2}\langle\Omega_{*}(I-I^{0}),I-I^{0}\rangle + P_{*}(I-I^{0},\theta)
\end{eqnarray}
obtained in Theorem \ref{0*}.\ For a fixed nonzero $ I^0 $,\ there exists a full dimensional KAM torus
$ \mathcal{T}({I^{0}})= \mathbb{T}^{\mathbb{Z}} \times \{I=I^{0}\}$ with frequency $ \omega(I^{0})$ for the Hamiltonian $ H_{*}$,\ and $ \mathcal{T}: = \Psi^{-1}_{*}\mathcal{T}({I^{0}})$ is a full dimensional KAM torus for the initial Hamiltonian $ H$.
\end{corollary}
\begin{remark}
Theorem \ref{0*} and corollary \ref{1*} extend the original Kolmogorov theorem \cite{Kol1954On} to the infinite-dimensional Hamiltonian systems of short range.\ It is still open to extend the Kolmogorov theorem to some Hamiltonian PDEs,\ which is even thought to be a harder problem.\ See Bourgain \cite{bo2005on}.
\end{remark}
\begin{remark}
The initial conditions in Theorem \ref{0*} are strongly localized in space.\ More exactly,\ we define
\begin{eqnarray}
\nonumber |I^{0}_{j}| \lessdot  e^{|j|^{1+\alpha}},
\end{eqnarray}
for $ \alpha > 0 $,\ namely the decay is super-exponential.\ With this fast decay and the fact that the interaction starts with five orders,\ one can choose a length scale ($\mathbb{L}^{k}$) such that we only need to consider the finite Hamiltonian $ H(\theta,I(k))$ in $ k$-th KAM iteration.\ In fact,\ the methods of  Fr\"{o}hlich-Spencer-Wayne \cite{1986An} and P\"{o}schel \cite{P1990Small} are still valid for the Hamiltonian (\ref{maineq}). See Section 4 for more details of the super-exponential decay.
\end{remark}

The rest of the paper consists almost entirely of the proofs of the preceding results,\ which employs the usual Newton type iteration procedure to handle small divisor problems.\ In section 2 the corresponding homological equation is considered,\ and in section 3 one step of iterative scheme is described in details.\ The iteration itself takes place in section 4,\ and section 5 provides the estimate of measure.\ In section 6,\ we prove the main theorem.

\section{The Homological Equations}
\subsection{Derivation of homological equations}
The proof of main Theorem employs the rapidly
converging iteration scheme of Newton type to deal with small divisor problems
introduced by Kolmogorov, involving the infinite sequence of coordinate transformations.\ Recalling the sequence of length scales,\ $ \mathbb{L}^{k}\nearrow \infty $,\ we construct this transformation inductively,\ attempting,\ at the $ k$-th stage of the inductive process,\ to "kill" only those part of the interaction,\ $ f_{\lfloor i,j\rceil}$,\ with the point $ i $ and $j$ lying inside the box $ B_{\mathbb{L}^{k}} $,\ which consists of all sites $ j$ with $ |j|<\mathbb{L}^{k}$.\ In the following,\ we denote $ |\cdot |$ the sup-norm for any matrices or vectors of finite order.\ If $ A_{1}$ and $ A_{2}$ are $ k \times k $ and $ l \times l $ matrices with $ k < l $ respectively,\ we define
$$
A_{1} + A_{2} = \begin{pmatrix}
A_{1}&0\\
0&0
\end{pmatrix}_{l\times l} + A_{2}.
$$

In order to provide a formal statement,\ let us define precisely the analytic function $ f_{\lfloor i,j\rceil}(I,\theta)$.\ Assume $ f_{( i,j)}(I,\theta) = \sum_{\alpha_{i}+\alpha_{j}\geq 5}\widetilde{a}_{\alpha_{i}\alpha_{j}}(\theta_{i},\theta_{j})(I_{i})^{\alpha_{i}}(I_{j})^{\alpha_{j}}$,\ for $ \alpha_{j}\in \mathbb{N}$.\ We thus have
\begin{eqnarray}
\nonumber  f_{\lfloor i,j\rceil}(I,\theta)
 &=& \sum_{\alpha_{j-1}+\alpha_{j}\geq 5}\widetilde{a}_{\alpha_{j-1}\alpha_{j}}(\theta_{j-1},\theta_{j})(I_{j-1})^{\alpha_{j-1}}(I_{j})^{\alpha_{j}} + \sum_{\alpha_{j}\geq 5}\widetilde{a}_{\alpha_{j}}(\theta_{j})(I_{j})^{\alpha_{j}}\\
\nonumber &&+
\sum_{\alpha_{j}+\alpha_{j+1}\geq 5}\widetilde{a}_{\alpha_{j}\alpha_{j+1}}(\theta_{j},\theta_{j+1})(I_{j})^{\alpha_{j}}(I_{j+1})^{\alpha_{j+1}}.
\end{eqnarray}
Since $ I_{j}^{0} \neq 0 $ for some $ j\in \mathbb{Z}$,\ we can translate it to the zero point by taking a symplectic transformation $ \Phi$ given by
\[
(I,\theta) \rightarrow \begin{cases}
\rho = I-I^{0}\\
\theta=\theta.
\end{cases}
\]
Hence,\ (\ref{maineq}) has the form
\begin{eqnarray}
\nonumber H(\rho,\theta) &=& H_{0}(I^{0}+\rho) + \epsilon H_{1}(I^{0}+\rho,\theta)\\
\nonumber &=& H_{0}(I^{0}) + \sum_{j\in \mathbb{Z}}\omega_{j}(I^{0})\rho_{j} + \frac{1}{2}\sum_{j\in \mathbb{Z},|i-j|\leq1}\frac{\partial^{2}H_{0}(I^{0})}{\partial I_{i}\partial I_{j}}\rho_{i}\rho_{j}\\
\nonumber&&+ \sum_{\lfloor i,j\rceil|j\in \mathbb{Z}}V_{\lfloor i,j\rceil}(\rho) + \epsilon \tilde{H}_{1}(\rho,\theta)\\
\nonumber &=& e + \langle\omega,\rho\rangle + \frac{1}{2}\langle\Omega\rho,\rho\rangle + V(\rho) + \epsilon \tilde{H}_{1}(\rho,\theta),
\end{eqnarray}
where $ V(\rho)=   H_{0}(I^{0}+\rho)-(e + \langle\omega,\rho\rangle + \frac{1}{2}\langle\Omega\rho,\rho\rangle)= \underset{\lfloor i,j\rceil|j\in \mathbb{Z}}{\sum}V_{\lfloor i,j\rceil}(\rho)$ and
\begin{eqnarray}
\label{00*}\tilde{H}_{1}(\rho,\theta) &=& \sum_{\lfloor i,j\rceil}f_{\lfloor i,j\rceil}(\rho_{i}+I^0_{i},\rho_{j}+I^{0}_{j},\theta_{i},\theta_{j}) \\
\nonumber&=&\sum_{j\in \mathbb{Z}}f_{( j-1,j)}(\rho,\theta)+ f_{( j,j)}(\rho,\theta)+ f_{( j,j+1)}(\rho,\theta).
\end{eqnarray}
Correspondingly,\ the new domain turns to be
\[
{\mathcal{D}}_{s,r}:=\{(\rho,\theta)\in \mathbb{C}^{\mathbb{Z}} \times \mathbb{T}^{\mathbb{Z}}: \| \rho \| < s, |Im \theta|_{\infty} < r\}.
\]
For the perturbation $ \tilde{H}_{1}(\rho,\theta)$ with new variables $ (\rho,\theta)$,\ we easily have
\begin{eqnarray}
\nonumber &&f_{( j-1,j)}(\rho,\theta): \\
\nonumber &=& \sum_{\alpha_{j-1}+\alpha_{j}\geq 5}\widetilde{a}_{\alpha_{j-1}\alpha_{j}}(\theta_{j-1},\theta_{j})(\rho_{j-1}+I^{0}_{j-1})^{\alpha_{j-1}}
(\rho_{j}+I^{0}_{j})^{\alpha_{j}}\\
\nonumber&=&\sum_{\alpha_{j-1}+\alpha_{j}\geq 5}\widetilde{a}_{\alpha_{j-1}\alpha_{j}}(\theta_{j-1},\theta_{j})
\left(\sum^{\alpha_{j-1}}_{k=0}\mathbf{C}_{\alpha_{j-1}}^{k}\rho_{j-1}^{k}(I^{0}_{j-1})^{\alpha_{j-1}-k}\right)
\left(\sum^{\alpha_{j}}_{l=0}\mathbf{C}_{\alpha_{j}}^{l}\rho_{j}^{l}(I^{0}_{j})^{\alpha_{j}-l}\right),
\end{eqnarray}
\begin{eqnarray}
\nonumber f_{( j,j)}(\rho,\theta)&=& \sum_{\alpha_{j}\geq 5}\widetilde{a}_{\alpha_{j}}(\theta_{j})
(\rho_{j}+I^{0}_{j})^{\alpha_{j}}=\sum_{\alpha_{j}\geq 5}\widetilde{a}_{\alpha_{j}}(\theta)
\left(\sum^{\alpha_{j}}_{l=0}\mathbf{C}_{\alpha_{j}}^{k}\rho_{j}^{l}(I^{0}_{j})^{\alpha_{j}-k}\right),
\end{eqnarray}
and
\begin{eqnarray}
\nonumber &&f_{(j,j+1)}(\rho,\theta)\\
\nonumber &=&
\sum_{\alpha_{j}+\alpha_{j+1}\geq 5}\widetilde{a}_{\alpha_{j}\alpha_{j+1}}(\theta_{j},\theta_{j+1})(\rho_{j}+I^{0}_{j})^{\alpha_{j}}
(\rho_{j+1}+I^{0}_{j+1})^{\alpha_{j+1}}\\
\nonumber &=&
\sum_{\alpha_{j}+\alpha_{j+1}\geq 5}\widetilde{a}_{\alpha_{j}\alpha_{j+1}}(\theta_{j},\theta_{j+1})
\left(\sum^{\alpha_{j}}_{k=0}\mathbf{C}_{\alpha_{j}}^{k}\rho_{j}^{k}(I^{0}_{j})^{\alpha_{j}-k}\right)
\left(\sum^{\alpha_{j+1}}_{l=0}\mathbf{C}_{\alpha_{j+1}}^{l}\rho_{j+1}^{l}(I^{0}_{j+1})^{\alpha_{j+1}-l}\right).
\end{eqnarray}
Moreover,\ the analytic function $ \tilde{H}_{1}(\rho,\theta)$ can be expanded into power series
\[
\tilde{H}_{1}(\rho,\theta) = \tilde{H}^{\theta}_{1}(\theta) + \langle \tilde{H}^{\rho}_{1}(\theta), \rho\rangle + \langle \tilde{H}^{\rho\rho}_{1}(\theta)\rho, \rho\rangle + \sum_{j\in \mathbb{Z},l\geq 3} {W}^{l}_{j}(\rho),
\]
where
$ W(\rho)= \underset{j\in \mathbb{Z},l\geq 3}{\sum}W^{l}_{j}(\rho)$ with $ W^{l}_{j}(\rho)= \mathcal{O}(|\rho_{j-1}|^{l_{j-1}}|\rho_{j}|^{l_{j}}) + \mathcal{O}(|\rho_{j}|^{l_{j}}|\rho_{j+1}|^{l_{j+1}})$ of $ l_{j-1}+l_{j},l_{j}+l_{j+1} \geq 3$.

From (\ref{00*}),\ one has
\begin{eqnarray}
\nonumber\tilde{H}^{\theta}_{1}(\theta)
&=& \sum_{j\in \mathbb{Z}}\sum_{\alpha_{j-1}+\alpha_{j}\geq 5}{a}_{\alpha_{j-1}\alpha_{j}}(\theta_{j-1},\theta_{j})
(I^{0}_{j-1})^{\alpha_{j-1}}(I^{0}_{j})^{\alpha_{j}}+\sum_{\alpha_{j}\geq 5}{a}_{\alpha_{j}}(\theta_{j})(I^{0}_{j})^{\alpha_{j}}\\
\nonumber &&+\sum_{\alpha_{j}+\alpha_{j+1}\geq 5}{a}_{\alpha_{j}\alpha_{j+1}}(\theta_{j},\theta_{j+1})
(I^{0}_{j})^{\alpha_{j}}(I^{0}_{j+1})^{\alpha_{j+1}};\\
\nonumber\tilde{H}^{\rho_{j}}_{1}(\theta)&=&
\sum_{\alpha_{j-1}+\alpha_{j}\geq 5}{a}_{\alpha_{j-1}\alpha_{j}}(\theta_{j-1},\theta_{j})
(I^{0}_{j-1})^{\alpha_{j-1}}(I^{0}_{j})^{\alpha_{j}-1}+\sum_{\alpha_{j}\geq 5}{a}_{\alpha_{j}}(\theta_{j})(I^{0}_{j})^{\alpha_{j}-1}\\
\nonumber &&+\sum_{\alpha_{j}+\alpha_{j+1}\geq 5}{a}_{\alpha_{j}\alpha_{j+1}}(\theta_{j},\theta_{j+1})
(I^{0}_{j})^{\alpha_{j}-1}(I^{0}_{j+1})^{\alpha_{j+1}};\\
\nonumber\tilde{H}^{\rho_{j-1}\rho_{j}}_{1}(\theta)&=& \tilde{H}^{\rho_{j}\rho_{j-1}}_{1}(\theta)=
\frac{1}{2}(\sum_{\alpha_{j}+\alpha_{j-1}\geq 5}{a}_{\alpha_{j}\alpha_{j-1}}(\theta_{j},\theta_{j-1})
(I^{0}_{j})^{\alpha_{j}-1}(I^{0}_{j-1})^{\alpha_{j-1}-1}\\
\nonumber &&+\sum_{\alpha_{j-1}+\alpha_{j}\geq 5}{a}_{\alpha_{j-1}\alpha_{j}}(\theta_{j-1},\theta_{j})
(I^{0}_{j-1})^{\alpha_{j-1}-1}(I^{0}_{j})^{\alpha_{j}-1});\\
\nonumber\tilde{H}^{\rho_{j}\rho_{j}}_{1}(\theta)&=&
\sum_{\alpha_{j-1}+\alpha_{j}\geq 5}{a}_{\alpha_{j-1}\alpha_{j}}(\theta_{j-1},\theta_{j})
(I^{0}_{j-1})^{\alpha_{j-1}}(I^{0}_{j})^{\alpha_{j}-2}+\sum_{\alpha_{j}\geq 5}{a}_{\alpha_{j}}(\theta_{j})(I^{0}_{j})^{\alpha_{j}-2}\\
\nonumber &&+\sum_{\alpha_{j}+\alpha_{j+1}\geq 5}{a}_{\alpha_{j}\alpha_{j+1}}(\theta_{j},\theta_{j+1})
(I^{0}_{j})^{\alpha_{j}-2}(I^{0}_{j+1})^{\alpha_{j+1}};\\
\nonumber\tilde{H}^{\rho_{i}\rho_{j}}_{1}(\theta)&=& 0,\   \mbox{ for any}\ |i-j|> 1.
\end{eqnarray}

Now consider the Hamiltonian  $ H $ of the form
\begin{eqnarray}\label{N1*}
{H} &=& N + R,
\end{eqnarray}
where
\begin{eqnarray}
\nonumber N(\rho) &=& e + \langle\omega,\rho\rangle + \frac{1}{2}\langle\Omega \rho,\rho\rangle,
\end{eqnarray}
and
\begin{eqnarray}
\nonumber
{R}(\rho,\theta) &=& V(\rho) + \epsilon\tilde{H}_{1}(\rho,\theta).
\end{eqnarray}
More precisely,\ let
\begin{eqnarray}
\nonumber
V(\rho,\theta) &=& {V}_{1}(\rho)+ {V}_{2}(\rho),
\end{eqnarray}
with
\begin{eqnarray}
\nonumber
{V}_{1}(\rho) &=& \sum_{\lfloor i,j\rceil|\ |i|,|j|\leq\mathbb{L}^+-1}V_{\lfloor i,j\rceil}(\rho),\\
\nonumber
{V}_{2}(\rho) &=& \sum_{\lfloor i,j\rceil|\ |i|\ \mbox{or}\ |j|\geq\mathbb{L}^+}V_{\lfloor i,j\rceil}(\rho);
\end{eqnarray}
and
\begin{eqnarray}
\nonumber
\epsilon\tilde{H}_{1}(\rho,\theta) &=& R_{1}(\rho,\theta) + R_{2}(\rho,\theta),
\end{eqnarray}
with
\begin{eqnarray}
\nonumber
{R}_{1}(\rho,\theta) &=& \epsilon\sum_{\lfloor i,j\rceil|dist(\lfloor i,j\rceil,0)\leq\mathbb{L}^{+}-1}f_{\lfloor i,j\rceil}(I^{0}_{i}+\rho_{i},I^{0}_{j}+\rho_{j},\theta_{i},\theta_{j}),\\
\nonumber
{R}_{2}(\rho,\theta) &=& \epsilon\sum_{\lfloor i,j\rceil|dist(\lfloor i,j\rceil,0)\geq\mathbb{L}^{+}}f_{\lfloor i,j\rceil}(I^{0}_{i}+\rho_{i},I^{0}_{j}+\rho_{j},\theta_{i},\theta_{j}).
\end{eqnarray}
Denote $
 R_{1}(\rho,\theta) =  R_{1}^{low}(\rho,\theta)+ R_{1}^{high}(\rho,\theta)$,\ where
 \begin{eqnarray}
 \label{030}R_{1}^{low}(\rho,\theta)&=&  R_{1}^{0}(\theta)+ \langle  R_{1}^{1}(\theta),\rho\rangle + \langle R_{1}^{2}(\theta)\rho,\rho\rangle,
\end{eqnarray}
and
\begin{eqnarray}
 \label{031} R_{1}^{high}(\rho,\theta)=  R_{1}^{3}(\rho,\theta)+ R_{1}^{4}(\rho,\theta),
\end{eqnarray}
with $ R_{1}^{3}(\rho,\theta)$ being cubic of $ |\rho_{j}|$ and $ R_{1}^{4}(\rho,\theta) = \mathcal{O}(|\rho_{j}|^{l})$ for $ l \geq 4$.

We desire to eliminate the terms $ R_{1}^{low} $ in (\ref{030}) by the coordinate transformation $\Psi$, which is obtained as the time-1- map $X_F^{t}|_{t=1}$ of a Hamiltonian
vector field $X_F$.\ Write
\begin{eqnarray}
\nonumber F(\rho,\theta) &=& \tilde{F}(\rho,\theta)+ \langle a,\theta\rangle =F^0(\theta) + \langle F^1(\theta),\rho\rangle + \langle F^2(\theta)\rho,\rho\rangle + \langle a,\theta\rangle,
\end{eqnarray}
where vector $ a$ is chosen to keep the frequency $ \omega$ fixed.
\begin{remark}
 Since the existing of the term $ \langle a,\theta\rangle $,\ the function $ F$ is defined on $ \mathbb{C}^{\mathbb{Z}} \times \mathbb{C}^{\mathbb{Z}}$ not $ \mathbb{C}^{\mathbb{Z}} \times \mathbb{T}^{\mathbb{Z}} $,\ but the flow $ X_{F}$ is perfectly well defined on $ \mathbb{C}^{\mathbb{Z}} \times \mathbb{T}^{\mathbb{Z}}$.
\end{remark}
Using Taylor's formula,\ we have
\begin{eqnarray}\label{013*}
H_{+}&=&H \circ\Psi= H \circ X_F^{t}|_{t=1}\\
\nonumber &=& H + \{H,F\} + \int_{0}^1(1-t)\{\{H,F\},F\}\circ X_F^{t}\ \mathrm{d}{t}\\
\nonumber &=& N + \{N,F\} + \int_{0}^1(1-t)\{\{N,F\},F\}\circ X_F^{t}\ \mathrm{d}{t}\\
\nonumber &&+ R_{1}^{low} + \int_{0}^1\{R_{1}^{low},F\}\circ X_F^{t}\ \mathrm{d}{t}\\
\nonumber &&+ R_{1}^{high}+ {V}_{1} + \{R_{1}^{high}+ {V}_{1},F\}+  ({V}_{2} + R_{2})\circ X_F^{1}\\
\nonumber && + \int_{0}^1(1-t)\{\{R_{1}^{high}+ {V}_{1},F\},F\}\circ X_F^{t}\ \mathrm{d}{t}.
\end{eqnarray}
Then we obtain the modified homological equation
\begin{eqnarray}\label{N2*}
\{{N},{F}\} + R_{1}^{low} + \{R_{1}^{high}+{V}_{1},F\}^{low}= N_{+} - N,
\end{eqnarray}
where $ N_{+}$ is given below and $ \{\cdot,\cdot\}$ is the Poisson bracket of functions on $ \mathcal{D}_{s,r}$ computed by the formula
\begin{eqnarray}\label{N2*+}
\{ F,G\} = \langle F_{\theta},G_{\rho}\rangle - \langle F_{\rho},G_{\theta} \rangle.
\end{eqnarray}\\
If the  homological equation (\ref{N2*}) is solved,\ the new perturbation term $ R_{1+} $ can be written as
\begin{eqnarray}
\label{032} R_{1+} &=& R_{1}^{high}+ {V}_{1} + \{R_{1}^{high}+ {V}_{1},F\}^{high}\\
\label{033} &&+ \int_{0}^1(1-t)\{\{N + R_{1}^{high}+ {V}_{1} ,F\},F\}\circ X_F^{t}\ \mathrm{d}{t}\\
\label{034} &&+ \int_{0}^1\{R_{1}^{low},F\}\circ X_F^{t}\ \mathrm{d}{t}.
\end{eqnarray}
Note that we do not need to eliminate the terms in (\ref{032}) at the next step of KAM procedure,\ so (\ref{032}) is not necessary to be too small.\ On the other hand,\ the remaining terms are either quadratic in $ F $ or bounded by $ \epsilon^{1+\beta}$(we will prove this in details below).\ Therefore,\ we can obtain a non-degenerate normal form of order 2 with a fixed frequency $ \omega$.

Once all the above procedures work well,\ our new Hamiltonian reads
 \[
 H_{+} = N_{+} +  R_{1+} + {V}_{2} + R_{2}.
 \]
Note $ R_{2} = R_{21} + R_{22}$,\ where
\begin{eqnarray}
\nonumber
{R}_{21}(\rho,\theta) &=& \epsilon\sum_{\lfloor i,j\rceil|\mathbb{L}^{+}\leq dist(\lfloor i,j\rceil,0)\leq\mathbb{L}^{++}-1}f_{\lfloor i,j\rceil}(I^{0}_{i}+\rho_{i},I^{0}_{j}+\rho_{j},\theta_{i},\theta_{j}),\\
\nonumber
{R}_{22}(\rho,\theta) &=& \epsilon\sum_{\lfloor i,j\rceil|dist(\lfloor i,j\rceil,0)\geq\mathbb{L}^{++}}f_{\lfloor i,j\rceil}(I^{0}_{i}+\rho_{i},I^{0}_{j}+\rho_{j},\theta_{i},\theta_{j}),
\end{eqnarray}
with a larger number $ \mathbb{L}^{++}$.

It is clear that one has to eliminate the term $ R_{1+} + R_{21}$ in the next iterative process and from this term one easily gets that $R_{1+}$ depends only on $ \rho_{j},\theta_{j}$ for $ |j|\leq\mathbb{L}^{+}$,\ while $ R_{21}$ depends on $ \rho_{j},\theta_{j}$ for $ |j|\geq \mathbb{L}^{+}$.\ So in order to know exactly what the new error term depends on and keep the eliminated term unified from $ 1$-th iteration($ R_{1}$ is of the form like $ A + B $),\ we will rewrite $ R_{1}$ into several terms directly below.\\
Let
\begin{eqnarray}
\nonumber
{R}_{1} &=& P(\rho,\theta) + \epsilon Q(\rho,\theta)\\
\label{035} &&+\epsilon\sum_{\mbox{either}\ |i|\ \mbox{or}\ |j|\ \geq\ \mathbb{L}^+\ \mbox{but not both}}f_{\lfloor i,j\rceil}(I^{0}_{i}+\rho_{i},I^{0}_{j}+\rho_{j},\theta_{i},\theta_{j}),
\end{eqnarray}
with
\begin{eqnarray}\label{N1**}
\nonumber P(\rho,\theta) &=&  \epsilon\left(f_{\lfloor 0,0\rceil}(I_{0}^{0}+\rho_{0},\theta_{0})+\sum_{|j|=\mathbb{L}=1}f_{\lfloor 0,j\rceil}(I_{0}^{0}+\rho_{0},I_{j}^{0}+\rho_{j},\theta_{0},\theta_{j})\right),
\end{eqnarray}
and
\begin{eqnarray}
\nonumber Q (\rho,\theta)&=& \sum_{\lfloor i,j\rceil|1=\mathbb{L}\leq |i|,|j|\leq \mathbb{L}^+-1}f_{\lfloor i,j\rceil}(I^{0}_{i}+\rho_{i},I^{0}_{j}+\rho_{j},\theta_{i},\theta_{j}).
\end{eqnarray}
\begin{remark}\label{*} Since $ \mathbb{L}=1$ and $ |f_{\lfloor i,j\rceil}(I,\theta)|\leq K\exp[-l(|j|-1)^{1+\alpha}]$,\ we can choose $ K=(2^{6})^{-1} $ such that $ |P|_{s,r} \leq \epsilon$.
\end{remark}
More exactly,\ denote
\[
N = N ^{0}+ N^{1} + N^{2},
\]
where
\begin{eqnarray}
\nonumber N^{0} &=& \underset{|j|\leq\mathbb{L}^+-1}{\sum}\omega_{j}\rho_{j}+ \underset{|i|,|j|\leq\mathbb{L}^+-1}{\sum}\Omega_{ij}\rho_{i}\rho_{j},\\
\nonumber N^{1} &=& \underset{\mbox{either}\ |i|\ \mbox{or}\ |j|=\mathbb{L}^+,\ \mbox{but not both}}{\sum}\Omega_{ij}\rho_{i}\rho_{j},\\
\nonumber N^{2}&=& \underset{|j|\geq\mathbb{L}^{+}}{\sum}\omega_{j}\rho_{j}+ \underset{|i|,|j|\geq\mathbb{L}^+}{\sum}\Omega_{ij}\rho_{i}\rho_{j};
\end{eqnarray}
and denote
\[
{V}(\rho) = \tilde{V}(\rho)+ \bar{V}(\rho) +\check{V}(\rho),
\]
where
\begin{eqnarray}
\nonumber
\tilde{V}(\rho) &=& \sum_{\lfloor i,j\rceil|\ |i|,|j|\leq\mathbb{L}^+-1}V_{\lfloor i,j\rceil}(\rho),\\
\nonumber
\bar{V}(\rho) &=& \sum_{\lfloor i,j\rceil|\ |i|\ \mbox{or}\ |j|=\mathbb{L}^+}V_{\lfloor i,j\rceil}(\rho),\\
\nonumber
\check{V}(\rho) &=& \sum_{\lfloor i,j\rceil|\ |i|,|j|\geq\mathbb{L}^+}V_{\lfloor i,j\rceil}(\rho).
\end{eqnarray}

For convenience,\ denote $
 P(\rho,\theta) = P^{low}(\rho,\theta)+P^{high}(\rho,\theta)$,\ where
 \begin{eqnarray}
P^{low}(\rho,\theta)&=& P^{0}(\theta)+ \langle P^{1}(\theta),\rho\rangle + \langle P^{2}(\theta)\rho,\rho\rangle,
\end{eqnarray}
and
\begin{eqnarray}
P^{high}(\rho,\theta)= P^{3}(\rho,\theta)+P^{4}(\rho,\theta),
\end{eqnarray}
with $ P^{3}(\rho,\theta)$ being cubic of $ |\rho_{j}|$ and $ P^{4}(\rho,\theta) = \mathcal{O}(|\rho_{j}|^{l})$ for $ l \geq 4$.\\
Similarly,\ note $Q(\rho,\theta) = Q^{low}(\rho,\theta)+Q^{high}(\rho,\theta)$,\ where
\begin{eqnarray}
Q^{low}(\rho,\theta)&=& Q^{0}(\theta)+ \langle Q^{1}(\theta),\rho\rangle + \langle Q^{2}(\theta)\rho,\rho\rangle,
\end{eqnarray}
and
\begin{eqnarray}
Q^{high}(\rho,\theta)= Q^{3}(\rho,\theta)+Q^{4}(\rho,\theta),
\end{eqnarray}
with $ Q^{3}(\rho,\theta) $ being cubic of $ |\rho_{j}|$ and $ Q^{4}(\rho,\theta) = \mathcal{O}(|\rho_{j}|^{l})$ for $ l \geq 4$.\\
Particularly,\ $ \tilde{V}(\rho)$ has the same form
\[
\tilde{V}(\rho) = \tilde{V}^{3}(\rho)+ \tilde{V}^{4}(\rho),
\]
with $ \tilde{V}^{3}(\rho) $ is cubic of $ |\rho_{j}|$ and $ \tilde{V}^{4}(\rho,\theta) = \mathcal{O}(|\rho_{j}|^{l})$ for $ l \geq 4$;\
\[
\bar{V}(\rho)=\bar{V}^{3}(\rho)+\bar{V}^{4}(\rho),
\]
with $ \bar{V}^{3}(\rho) $ being cubic of $ |\rho_{j}|$ and $ \bar{V}^{4}(\rho,\theta) = \mathcal{O}(|\rho_{j}|^{l})$ for $ l \geq 4$;\ and
\[
\check{V}(\rho)=\check{V}^{3}(\rho)+\check{V}^{4}(\rho),
\]
with $ \check{V}^{3}(\rho) $ being cubic of $ |\rho_{j}|$ and $ \check{V}^{4}(\rho,\theta) = \mathcal{O}(|\rho_{j}|^{l})$ for $ l \geq 4$.

Hence,\ the term we desire to eliminate is $ P^{low}+\epsilon Q^{low}$ in (\ref{N2*}) since (\ref{035}) can be bounded by $\epsilon^{1+\beta}$(see below).\ Furthermore,\ in order to cure the problem that the measures of $ \omega $ in $(\mathcal{A}_{1})$ will diverge,\ we reduce the infinite sum of $ v$ to a finite one.\ In other words, the homological equation of (\ref{N2*}) turns to be
\begin{eqnarray}\label{N2}
\{{N}^{0},{F}\} + \Gamma P^{low}+\epsilon \Gamma Q^{low}+ \{\tilde{V} &+&\Gamma P^{high}+\epsilon \Gamma Q^{high},F\}^{low}\\
\nonumber&=& N^{0}_{+} - N^{0},
\end{eqnarray}
where $\Gamma P, \Gamma Q$ mean the truncation of the $ M^{+}$-th Fourier coefficient of $ P, Q$\ with $ M^{+}$ given later.\\
Let ${F}^{0}$( resp. $ F^1,F^2$)  has the form of $P^0+\epsilon Q^{0}$( resp.$ P^1+\epsilon Q^{1},P^2+\epsilon Q^{2}$).\
That is,\  ${F}^{0}$( resp. $ F^1,F^2$) can be expanded to
\begin{eqnarray}
\label{aa}{F}^0(\theta)&=&\sum_{0< |v| < M^{+}}{\widehat{F}_{v}}^{0}e^{\mathbf{i}\langle v,\theta\rangle},
\end{eqnarray}
where $ {\widehat{F}_{v}}^{j}$ are the $ v$-th Fourier coefficients of $ {F}^{j}$($ j= 0,1,2$).\\
Note $ \langle U\rho,\rho\rangle
=\{\tilde{V}+ P^{high}+\epsilon  Q^{high},F\}^{low}=\{\tilde{V}^{3}+\Gamma P^{3}+\epsilon \Gamma Q^{3},F^{0}+\langle a,\theta\rangle\} $.\ From the definition of Poisson bracket (\ref{N2*+}),\ we have the following equations from distinguishing terms of (\ref{N2}) by the order of $ \rho $:
\begin{equation}\label{**}
\begin{cases}
-\partial_{\omega}F^{0}-\langle \omega(+), a\rangle + \Gamma{P}^{0}+ \epsilon\Gamma{Q}^{0}=0,\\
-\partial_{\omega}F^{1}-\tilde{\Omega} a -\tilde{\Omega}\partial_{\theta}F^{0}+\Gamma{P}^{1}+ \epsilon\Gamma{Q}^{1}=0,\\
-\partial_{\omega}F^{2}-\tilde{\Omega}\partial_{\theta}F^{1}+ U +\Gamma{P}^{2}+ \epsilon\Gamma{Q}^{2}=0,
\end{cases}
\end{equation}
where $ \partial_{\omega}= \omega\cdot\partial_{\theta} $, $ \omega(+)= (\omega_{j})_{|j|\leq\ \mathbb{L}^{+}-1} $ and $ \tilde{\Omega}= {\Omega}(+)=(\Omega_{ij})_{|i|,|j|\leq \mathbb{L}^{+}-1}$.

Therefore,\ the new Hamiltonian ${H}_{+}$ has the form
\begin{eqnarray}\label{013}
H_{+}&=&H\circ X_F^{t}|_{t=1}\\
&=&\nonumber {N}+\{{N},F\}+ \Gamma P^{low}+\epsilon \Gamma Q^{low}+\tilde{V}^{3}+\bar{V}^{3}+ \Gamma P^{3}+ \epsilon\Gamma Q^{3}\\
&&\nonumber+\int_{0}^1\{\Gamma P^{low}+\epsilon \Gamma Q^{low},F\}\circ X_F^{t}\ \mathrm{d}{t}\\
&&\nonumber+ \{\tilde{V}^{3}+\Gamma P^{3}+ \epsilon\Gamma Q^{3},F\}^{low}+ \{\tilde{V}^{3}+\Gamma P^{3}+ \epsilon\Gamma Q^{3},F\}^{3}\\
&&\nonumber+ \{\tilde{V}^{3}+\Gamma P^{3}+ \epsilon\Gamma Q^{3},F\}^{4}+ \{\bar{V}^{3},F\}\\
&&\nonumber+\int_{0}^1(1-t)\{\{{N}+\tilde{V}^{3}+\bar{V}^{3}+\Gamma P^{3}+ \epsilon\Gamma Q^{3},F\},F\}\circ X_F^{t}\mathrm{d}{t}\\
&&\nonumber+ \left(\sum_{|v|\geq M^{+}}\widehat{P}_{v}+\epsilon \widehat{Q}_{v}+\tilde{V}^{4}+\bar{V}^{4}+\Gamma P^{4}+ \epsilon\Gamma Q^{4}\right)\circ X_F^1+ \check{V}+ R_{2}
\\
&&\nonumber+ \left(\epsilon\sum_{\mbox{either}\ |i|\ \mbox{or}\ |j|\ \geq\ \mathbb{L}^+\ \mbox{but not both}}\ f_{\lfloor i,j\rceil}(I^0+\rho,\theta)\right)\circ X_F^1 \\
\nonumber&=&\label{N4}N_+ + P_+ + \check{V} + \epsilon\sum_{\lfloor i,j\rceil|dist(\lfloor i,j\rceil,0)\geq\mathbb{L}^{+}}f_{\lfloor i,j\rceil}(I^{0}_{i}+\rho_{i},I^{0}_{j}+\rho_{j},\theta_{i},\theta_{j}),
\end{eqnarray}
where
\begin{equation}\label{N5}
N_+ = \hat{e} +{N}  + [{P}^{2}] +\epsilon [{Q}^{2}]+ [{U}],
\end{equation}
with $ \hat{e} = [{P}^{0}]+\epsilon [{Q}^{0}] + \langle \omega(+), a\rangle $,\ the $ [{P}^{j}] + \epsilon[{Q}^{j}]$,\ $ [{U}] $ being respectively the $0$-th Fourier coefficients of $ {P}^{j}+ \epsilon{Q}^{j}$ $ (j= 0,1,2)$,\ ${U}$,\ and
\begin{eqnarray}\label{N6}
P_+ :&=&\tilde{V}^{3}+\bar{V}^{3}+\Gamma P^{3}+ \epsilon\Gamma Q^{3}+\tilde{V}^{4}+\bar{V}^{4}+\Gamma P^{4}+ \epsilon\Gamma Q^{4}\\
&&\nonumber+ \{\tilde{V}^{3}+\Gamma P^{3}+ \epsilon\Gamma Q^{3},F\}^{3}+\{\tilde{V}^{3}+\Gamma P^{3}+ \epsilon\Gamma Q^{3},F\}^{4}\\
&&\nonumber+ \{\tilde{V}^{4}+\Gamma P^{4}+ \epsilon\Gamma Q^{4},F\}^{3}+\{\tilde{V}^{4}+\Gamma P^{4}+ \epsilon\Gamma Q^{4},F\}^{4}\\
&&\nonumber+ \int_{0}^1\{\Gamma P^{low}+\epsilon \Gamma Q^{low},F\}\circ X_F^{t}\ \mathrm{d}{t}\\
&&\nonumber+\int_{0}^1(1-t)\{\{{N}+\tilde{V}+\bar{V}+\Gamma P^{high}+\epsilon\Gamma Q^{high} ,F\},F\}\circ X_F^{t}\ \mathrm{d}{t}
\\
&&\nonumber+ \left(\sum_{|v|\geq M^{+}}\widehat{P}_{v}+\epsilon \widehat{Q}_{v}\right)\circ X_F^1+ \{{N}^{1},F\} + \{\bar{V}^{3},F\}\\
&&\nonumber+ \left(\epsilon\sum_{\mbox{either}\ |i|\ \mbox{or}\ |j|\ \geq\ \mathbb{L}^+\ \mbox{but not both}}\ f_{\lfloor i,j\rceil}(I^0+\rho,\theta)\right)\circ X_F^1,
\end{eqnarray}
where $ P_+ $ depends only on $ \rho_{j}, \theta_{j}$ for $ |j|\leq \mathbb{L}^{+}$.

\subsection{The solvability of the homological equations}In this subsection, we will estimate
the solutions of the homological equations.\ To avoid a flood of constants we will write $ a\lessdot b$,\ if there exists a constant $ C \geq 1$ depending only on $ \alpha,\gamma$ such that $ a \leq Cb $.\  Moreover,\ one has the following estimates.\\
Recalling that
\begin{eqnarray}
\nonumber P^{low}(\rho,\theta)&=& P^{0}(\theta)+ \langle P^{1}(\theta),\rho\rangle + \langle P^{2}(\theta)\rho,\rho\rangle\\
\nonumber &=& \sum_{0\leq |j|\leq 1}P^{0\theta_{j}}(\theta)+ \sum_{0\leq |j|\leq 1}P^{1\rho_{j}}(\theta)\rho_{j} + \sum_{0\leq |j|\leq 1}P^{2\rho_{0}\rho_{j}}(\theta)\rho_{0}\rho_{j},
\end{eqnarray}
where $ P^{0\theta_{j}}(\theta) = P^{0\theta_{j}}(\theta_{j-1},\theta_{j},\theta_{j+1})$, one has the estimates
\begin{eqnarray}
\nonumber|P^{0\theta_{j}}(\theta)| &\leq & \epsilon,\\
\nonumber|P^{1\rho_{j}}(\theta)| &\leq & \epsilon,\\
\nonumber|P^{2\rho_{j}\rho_{j}}(\theta)| &\leq & \epsilon,\\
\nonumber|P^{2\rho_{0}\rho_{j}}(\theta)| &\leq & \epsilon
\end{eqnarray}
for $ 0\leq |j| \leq 1$.\\
Similarly,\ since
\begin{eqnarray}
\nonumber Q^{low}(\rho,\theta)&=& Q^{0}(\theta)+ \langle Q^{1}(\theta),\rho\rangle + \langle Q^{2}(\theta)\rho,\rho\rangle\\
\nonumber &=&\sum_{1\leq |j|\leq \mathbb{L}^+-1}Q^{0\theta_{j}}(\theta)+ \sum_{1\leq |j|\leq \mathbb{L}^+-1}Q^{1\rho_{j}}(\theta)\rho_{j} \\
\nonumber&&+ \sum_{\lfloor i,j\rceil|1\leq |i|,|j|\leq \mathbb{L}^+-1}Q^{2\rho_{i}\rho_{j}}(\theta)\rho_{i}\rho_{j},
\end{eqnarray}
where $ Q^{0\theta_{j}}(\theta) = Q^{0\theta_{j}}(\theta_{j-1},\theta_{j},\theta_{j+1})$,\ one obtains
\begin{eqnarray}
\nonumber|Q^{0\theta_{j}}(\theta)| &\leq & |I^{0}_{|j|-1}|^{5},\\
\nonumber|Q^{1\rho_{j}}(\theta)| &\leq & |I^{0}_{|j|-1}|^{4},\\
\nonumber|Q^{2\rho_{j}\rho_{j}}(\theta)| &\leq & |I^{0}_{|j|-1}|^{3},\\
\nonumber|Q^{2\rho_{i}\rho_{j}}(\theta)| &\leq& |I^{0}_{|j|-1}|^{3}, \mbox{for}\ |i-j|\leq1,
\end{eqnarray}
for $ 1\leq |i|,|j|\leq \mathbb{L}^+-1 $.
\numberwithin{equation}{section}
\begin{lemma}\label{N7*}
For the definition of the operator $ \bar{\Omega}(k): \mathbb{C}^{\mathbb{Z}}\rightarrow\mathbb{C}^{\mathbb{Z}}$ with the norm $ |||\cdot|||$ satisfying
\begin{eqnarray}\label{N6*}
|||\bar{\Omega}(k)|||\leq \kappa_{1}, |||\bar{\Omega}^{-1}(k)|||\leq \kappa_{2},\forall\ k,
\end{eqnarray}
we have
\begin{eqnarray}\label{N6**}
|{\Omega}(k)|\leq 2\kappa_{1}, |{\Omega}^{-1}(k)|\leq 2\kappa_{2},\forall\ k,
\end{eqnarray}
where $ |\cdot|$ denotes the sup-norm for any finite matrices.
\end{lemma}
\begin{proof}
For any given $k$,\ the matrix ${\Omega}(k)$ is symmetric by the definition of ${\Omega}(k)$.\ For any $j$,\
take $ I = (0,...,e^{-|j|^{1+\alpha}},0...)$ with $ ||I|| =0$.\ Then one has
\begin{eqnarray}
\nonumber||{\Omega}(k)I|| &=& |\Omega_{\{j-1\}j}|e^{|j-1|^{1+\alpha}-|j|^{1+\alpha}}+ |\Omega_{jj}|+|\Omega_{\{j+1\}j}|e^{|j+1|^{1+\alpha}-|j|^{1+\alpha}}\leq \kappa_{1},
\end{eqnarray}
which implies
\begin{eqnarray}
\nonumber |\Omega_{jj}|+|\Omega_{\{j+1\}j}|\leq \kappa_{1},\ j \geq 1,
\end{eqnarray}
or
\begin{eqnarray}
\nonumber |\Omega_{\{j-1\}j}|+|\Omega_{jj}|\leq \kappa_{1},\ j \leq -1.
\end{eqnarray}
Particularly, for $ j=0 $,\ we have
\begin{eqnarray}
\nonumber |\Omega_{\{-1\}0}|+|\Omega_{00}|+ |\Omega_{01}|\leq \kappa_{1}.
\end{eqnarray}
That is,\ one also has
\begin{eqnarray}
\nonumber |\Omega_{\{j-1\}\{j-1\}}|+|\Omega_{j\{j-1\}}|\leq \kappa_{1},\ j\geq 1,
\end{eqnarray}
and
\begin{eqnarray}
\nonumber |\Omega_{j\{j+1\}}|+|\Omega_{\{j+1\}\{j+1\}}|\leq \kappa_{1},\ j\leq -1.
\end{eqnarray}
Since $ \Omega_{j\{j-1\}} =\Omega_{\{j-1\}j}$ and $ \Omega_{j\{j+1\}}= \Omega_{\{j+1\}j}$,\ one finally obtains
\begin{eqnarray}
\nonumber |\Omega_{\{j-1\}j}|+|\Omega_{jj}|+|\Omega_{\{j+1\}j}|\leq 2\kappa_{1},\ \forall\ j.
\end{eqnarray}
Therefore,\ one has
\begin{eqnarray}\nonumber
|{\Omega}(k)|\leq 2\kappa_{1}.
\end{eqnarray}
The remaining proof is similar.
\end{proof}

\begin{lemma}\label{N7}
Given some positive parameters $ 0 < \beta < \frac{1}{10},\epsilon > 0$ and $ 0 < \tilde{\sigma} < s ,0 < \sigma < r$,\ one has
\[
|\{ F,G \}|_{s-\tilde{\sigma},r-\sigma}\leq 2\epsilon^{-\frac{1+\beta}{5}}\tilde{\sigma}^{-1}\sigma^{-1}|F|_{{s},r}|G|_{{s},r},
\]
where $ F$ and $ G $ are functions which are of the form $ G = G^{\theta} + \langle G^{\rho},\rho^{\ast} \rangle +\langle G^{\rho\rho}\rho^{\ast},\rho^{\ast}\rangle$ with $ \rho^{\ast} = (\rho_{j})_{|j| \leq \mathbb{L}^+-{1}}$ of $ \mathbb{L}^+ = (\frac{1+\beta}{5}|\ln \epsilon|)^{\frac{1}{1+\alpha}}$.
\end{lemma}
\begin{proof}
Let
\[
F = F^{\theta} + \langle F^{\rho},\rho^{\ast} \rangle +\langle F^{\rho\rho}\rho^{\ast},\rho^{\ast}\rangle
\]
and
\[
G = G^{\theta} + \langle G^{\rho},\rho^{\ast} \rangle +\langle G^{\rho\rho}\rho^{\ast},\rho^{\ast}\rangle.
\]
Recall that
\[
\{ F,G \} = \langle F_{\theta},G_{\rho} \rangle -\langle F_{\rho},G_{\theta} \rangle.
\]
Considering the term $ \langle F_{\theta},G_{\rho} \rangle $,\ we have
\begin{eqnarray}
\nonumber|\langle F_{\theta},G_{\rho} \rangle|_{s-\tilde{\sigma},r-\sigma}&=&\left|\sum_{j\in\mathbb{Z}} \frac{\partial F}{\partial\theta_{j}}\frac{\partial G}{\partial\rho_{j}}\right|_{s-\tilde{\sigma},r-\sigma}\\
\nonumber &\leq& \sum_{|j| \leq \mathbb{L}^+-{1}} \tilde{\sigma}^{-1}e^{|j|^{1+\alpha}}|G|_{{s},r-\sigma}|F_{\theta_{j}}|_{s-\tilde{\sigma},r-\sigma}\\
\nonumber&\leq& \tilde{\sigma}^{-1}e^{|\mathbb{L}^+-{1}|^{1+\alpha}}\sigma^{-1}|G|_{{s},r-\sigma}|F|_{s-\tilde{\sigma},r}\\
\nonumber&\leq& \epsilon^{-\frac{1+\beta}{5}}\tilde{\sigma}^{-1}\sigma^{-1}|G|_{{s},r-\sigma}|F|_{s-\tilde{\sigma},r}.
\end{eqnarray}
Similarly, one has
\begin{eqnarray}
\nonumber|\langle F_{\rho},G_{\theta} \rangle|_{s-\tilde{\sigma},r-\sigma}&=&\left|\sum_{j\in\mathbb{Z}} \frac{\partial F}{\partial\rho_{j}}\frac{\partial G}{\partial\theta_{j}}\right|_{s-\tilde{\sigma},r-\sigma}\\
\nonumber&\leq& \sum_{|j| \leq \mathbb{L}^+-{1}} \tilde{\sigma}^{-1}e^{|j|^{1+\alpha}}|F|_{{s},r-\sigma}|G_{\theta_{j}}|_{s-\tilde{\sigma},r-\sigma}\\
\nonumber&\leq& \tilde{\sigma}^{-1}e^{|\mathbb{L}^+-{1}|^{1+\alpha}}\sigma^{-1}|F|_{{s},r-\sigma}|G|_{s-\tilde{\sigma},r}\\
\nonumber&\leq& \epsilon^{-\frac{1+\beta}{5}}\tilde{\sigma}^{-1}\sigma^{-1}|F|_{{s},r-\sigma}|G|_{s-\tilde{\sigma},r}.
\end{eqnarray}
It follows that
\[
|\{ F,G \}|_{s-\tilde{\sigma},r-\sigma}\leq 2\epsilon^{-\frac{1+\beta}{5}}\tilde{\sigma}^{-1}\sigma^{-1}|F|_{{s},r}|G|_{{s},r}.
\]
\end{proof}

\begin{lemma}\label{N8}
Let $ \omega $ be Diophantine with $ \gamma > 0$ (see ($\mathcal{A}_{1}$) for $ k=1 $),\ and choose $ \mathbb{L}^+ = (\frac{1+\beta}{5}|\ln \epsilon|)^{\frac{1}{1+\alpha}} $.\ Then for any $ 0< s < 1 ,r > 0 $ and $ 0 < \sigma < \frac{1}{5}r $,\ the solutions of the homological equations which are given by (\ref{aa}),\ satisfy
\begin{eqnarray}
\label{014}|{F}^{0}|_{r-\sigma}&\lessdot& \frac{1}{\sigma^{2\mathbb{L}^+-1}}\epsilon^{1-\gamma},\\
\label{015}| F^{1}|_{r-3\sigma}&\lessdot& \frac{1}{\sigma^{4\mathbb{L}^+-1}}\epsilon^{1-2\gamma},\\
\label{016}| F^{2}|_{r-5\sigma}&\lessdot& \frac{1}{\sigma^{6\mathbb{L}^+-1}}\epsilon^{1-3\gamma}.
\end{eqnarray}
Moreover,\ one has
\begin{eqnarray}
|\tilde{F}|_{s,r-5\sigma} &\lessdot& \frac{1}{\sigma^{6\mathbb{L}^+-1}}\epsilon^{1-3\gamma},
\end{eqnarray}
where $ \tilde{F}= {F}^{0} + \langle F^{1},\rho\rangle + \langle F^{2}\rho,\rho\rangle$.
\end{lemma}
\begin{proof}
First of all,\ we consider the first equation of (\ref{**}).\ From (\ref{aa}),\ we can solve the equation
\[
\partial_{\omega}F^{0}= \Gamma{P}^{0}+ \epsilon\Gamma{Q}^{0}-([{P}^{0}] + \epsilon[{Q}^{0}] +\langle \omega(+), a\rangle ).
\]
That is,\ the solution of the homological equation is given by
\begin{eqnarray}
\label{009}F^{0}(\theta)&=&\sum_{0< |v| < M^{+}} \frac{\widehat{P}_{v}^{0}+\epsilon \widehat{Q}_{v}^{0}}{\mathbf{i}\langle \omega(+),v\rangle}e^{\mathbf{i}\langle v,\theta\rangle}.
\end{eqnarray}
From the Diophantine condition ($\mathcal{A}_{1}$),\ one has
\begin{eqnarray}
\nonumber|F^{0}(\theta)|_{r-\sigma} &\leq& \epsilon^{-\gamma}\sum_{0< |v| < M^{+}}\left(\epsilon +\epsilon \underset{|j|\leq \mathbb{L}^+-1}{\sum}|I^{0}_{|j|-1}|^{5}\right) e^{-r|v|}e^{(r-\sigma)|v|}\\
\nonumber &\lessdot& \frac{1}{\sigma^{2\mathbb{L}^+-1}}\epsilon^{1-\gamma},\  0 < \sigma < r,
\end{eqnarray}
which finishes the proof of (\ref{014}).\\
Moreover,\ combining with Cauchy estimate
\[
|\partial_{\theta}F^{0}(\theta)|_{r-2\sigma} \lessdot \frac{1}{\sigma}|F^{0}(\theta)|_{r-\sigma},
\]
where $ |\partial_{\theta}F^{0}(\theta)|_{r}= \underset{|j|\leq \mathbb{L}^+-1}{\max}\underset{\theta\in \mathbb{T}_{r}^{\mathbb{Z}}}{\sup}|\partial_{\theta_{j}}F^{0}(\theta)|$,\ one finally obtains
\begin{eqnarray}
\label{1010}|\partial_{\theta}F^{0}(\theta)|_{r-2\sigma}  &\lessdot& \frac{\epsilon^{1-\gamma}}{\sigma^{2\mathbb{L}^+}}.
\end{eqnarray}

Next we consider the second equation of (\ref{**}).\ Since $ \widehat{F}_{0}^{1}=0 $,\ we can choose a vector $ a $ such that $ \tilde{\Omega}a -[{P}^{1}]-\epsilon [{Q}^{1}]=0$.\ Then we have
\begin{equation}\label{012}
a = (\tilde{\Omega})^{-1}([{P}^{1}] +\epsilon [{Q}^{1}]),
\end{equation}
with the estimate
\begin{eqnarray}
\nonumber |a| &=& |(\tilde{\Omega})^{-1}||[{P}^{1}] +\epsilon [{Q}^{1}]|\\
\nonumber &\leq& 2{\kappa_{2}}(|[{P}^{1}]| +\epsilon |[{Q}^{1}]|)\\
\nonumber &\lessdot& \epsilon.
\end{eqnarray}
Let $ \tilde{F}^{0}= -\tilde{\Omega}\partial_{\theta}F^{0}$.\ We get
\[
|\tilde{F}^{0}(\theta)|_{r-2\sigma}\leq |\tilde{\Omega}||\partial_{\theta}F^{0}(\theta)|_{r-2\sigma}\lessdot \frac{1}{\sigma^{2\mathbb{L}^+}}\epsilon^{1-\gamma},
\]
where the last equality is based on (\ref{1010}).\\
Similarly to (\ref{009}),\ we can solve the equation
\begin{eqnarray}
\label{010}F^{1}(\theta)&=& \sum_{0< |v| < M^{+}}\frac{\widehat{P}_{v}^{1}+\epsilon \widehat{Q}_{v}^{1}+\widehat{\tilde{F}^{0}}_{v}}
{\mathbf{i}\langle \omega(+),v\rangle}e^{\mathbf{i}\langle v,\theta\rangle}.
\end{eqnarray}
and easily obtain the estimate
\begin{eqnarray}
\nonumber|F^{1}(\theta)|_{r-3\sigma} &\leq& \epsilon^{-\gamma}\sum_{0< |v| < M^{+}}\left(\epsilon +\frac{1}{\sigma^{2\mathbb{L}^+}}\epsilon^{1-\gamma}\right) e^{-(r-2\sigma)|v|}e^{(r-3\sigma)|v|}\\
\nonumber &\lessdot& \frac{1}{\sigma^{4\mathbb{L}^+-1}}\epsilon^{1-2\gamma},
\end{eqnarray}
which finishes the proof of (\ref{015}).\\
Furthermore,\ we have
\begin{eqnarray}
\label{1011}|\partial_{\theta}F^{1}(\theta)|_{r-4\sigma}\lessdot \frac{1}{\sigma^{4\mathbb{L}^+}}\epsilon^{1-2\gamma},
\end{eqnarray}
where $ |\partial_{\theta}F^{1}(\theta)|_{r}= \max_{i} \underset{|j|\leq \mathbb{L}^+-1}{\sum}\underset{\theta\in \mathbb{T}_{r}^{\mathbb{Z}}}{\sup}|\partial_{\theta_{j}}F^{1\rho_{i}}(\theta)|$.

Now we turn to the third equation of (\ref{**}).\ Let $ \tilde{F}^{1}= -\tilde{\Omega}\partial_{\theta}F^{1}$.\ On $ \mathcal{D}_{s-\tilde{\sigma},r-4\sigma}$,\  we have
\[
|\widetilde{F}^{1}(\theta)|_{r-4\sigma}\leq |\tilde{\Omega}||\partial_{\theta}F^{1}(\theta)|_{r-4\sigma}\lessdot \frac{1}{\sigma^{4\mathbb{L}^+}}\epsilon^{1-2\gamma}.
\]
Since
\[
\langle U\rho,\rho\rangle = \{\tilde{V}^{3} + \Gamma P^{3}+\epsilon \Gamma Q^{3},F^{0}+ \langle a,\theta\rangle\},
\]
and
\begin{eqnarray}
\nonumber \{\tilde{V}^{3},F^{0}+ \langle a,\theta\rangle\} &=& \sum_{i,j,l}\frac{\partial^{3}\tilde{V}^{3}}
{\partial \rho_{i}\partial \rho_{j}\partial \rho_{l}}(\partial_{\theta_{l}}F^{0}+ a_{l})\rho_{i}\rho_{j},
\end{eqnarray}
we easily estimate that
\begin{eqnarray}
\label{028} &&|\{\tilde{V}^{3},F^{0}+ \langle a,\theta\rangle\}^{\rho\rho}|\\
\nonumber &\leq& \sup_{j}\left(\left|\frac{\partial^{3}\tilde{V}^{3}}{\partial \rho_{j-1}^{2}\partial \rho_{j}}\right|
+\left|\frac{\partial^{3}\tilde{V}^{3}}{\partial \rho_{j}^{3}}\right|
+
\left|\frac{\partial^{3}\tilde{V}^{3}}{\partial \rho_{j+1}^{2}\partial \rho_{j}}\right|\right)
(|\partial_{\theta}F^{0}|+ |a|)\\
\nonumber &\lessdot & |\partial_{\theta}F^{0}|+ |a|
\end{eqnarray}
and
\begin{eqnarray}
\label{029} &&|\{\Gamma P^{3}+\epsilon \Gamma Q^{3},F^{0}+ \langle a,\theta\rangle\}^{\rho\rho}| \\
\nonumber&=& \max_i\sum_{|j|\leq \mathbb{L}^+-1}\left|\sum_{l}\frac{\partial^{3}(\Gamma P^{3}+\epsilon \Gamma Q^{3})}{\partial\rho_{i}\partial\rho_{j}\partial \rho_{l}}(\partial_{\theta_{l}}F^{0}+ a_{l})\right|\\
\nonumber &\leq& \sum_{|j|\leq \mathbb{L}^+-1}\epsilon|I^{0}_{|j|-1}|^{2}(|\partial_{\theta}F^{0}|+ |a|)\\
\nonumber &\lessdot& |\partial_{\theta}F^{0}|+ |a|.
\end{eqnarray}
From (\ref{028}) and (\ref{029}),\ one has
\begin{eqnarray}\label{036}
|U|_{r-2\sigma} &\lessdot& \frac{1}{\sigma^{2\mathbb{L}^+}}\epsilon^{1-\gamma}.
\end{eqnarray}
Similarly to (\ref{009}),\ the solution of the third equation of (\ref{**}) is given by
\begin{eqnarray}
\label{011}F^{2}(\theta)&=& \sum_{0< |v| < M^{+}}\frac{\widehat{P}_{v}^{2}+\epsilon \widehat{Q}_{v}^{2}+\widehat{\tilde{F}^{1}}_{v}+\widehat{U}_{v}}{\mathbf{i}\langle {\omega(+),v\rangle}}e^{\mathbf{i}\langle v,\theta\rangle},
\end{eqnarray}
and the corresponding estimate is
\begin{eqnarray}
\nonumber
|F^{2}(\theta)|_{r-5\sigma}\lessdot \frac{1}{\sigma^{6\mathbb{L}^+-1}}\epsilon^{1-3\gamma},
\end{eqnarray}
which finishes the proof of (\ref{016}).\\
Consequently,\ from (\ref{014}), (\ref{015}) and (\ref{016}),\ one obtains
\begin{eqnarray}
\nonumber &&|\tilde{F}|_{s,r-5\sigma}\\
 \nonumber &\leq& |F^{0}|_{r-5\sigma}+|F^{1}|_{r-5\sigma}\left(\underset{|j|\leq \mathbb{L}^+-1}{\sum}|\rho_{j}|\right)+|F^{2}|_{r-5\sigma}\left(\underset{|i|,|j|\leq \mathbb{L}^+-1}{\sum}|\rho_{i}||\rho_{j}|\right)\\
\nonumber&\lessdot& \frac{1}{\sigma^{2\mathbb{L}^+-1}}\epsilon^{1-\gamma}+ \frac{1}{\sigma^{4\mathbb{L}^+-1}}\epsilon^{1-2\gamma}\left(\underset{|j|\leq \mathbb{L}^+-1}{\sum}e^{-|j|^{1+\alpha}}\right)\\
\nonumber &&+
\frac{1}{\sigma^{6\mathbb{L}^+-1}}\epsilon^{1-3\gamma}\left(\underset{|j|\leq \mathbb{L}^+-1}{\sum}e^{-|j|^{1+\alpha}}\right)^{2}\\
\nonumber&\lessdot& \frac{1}{\sigma^{6\mathbb{L}^+-1}}\epsilon^{1-3\gamma}.
\end{eqnarray}
\end{proof}

\subsection{The derivatives of $ F $}

On $ D_{s,r-5\sigma}$,\ from Lemma \ref{N8},\ we obtain the estimate
\begin{eqnarray}
\label{050001*} |\tilde{F} |_{s,r-5\sigma}\lessdot \frac{1}{\sigma^{6\mathbb{L}^+-1}}\epsilon^{1-3\gamma}.
\end{eqnarray}
From the equation of motion
\begin{eqnarray}
\dot{\theta}= F_{\rho}(\rho,\theta), \dot{\rho}= F_{\theta}(\rho,\theta)= \tilde{F}_{\theta}(\rho,\theta)+ a,
\end{eqnarray}
vector $ a $ has to belong to $ \mathbb{C}^{\mathbb{Z}}$.\ That is,\ we have the estimate
\begin{eqnarray}
\label{1012}\| a\| &=& \sum_{|j|\leq \mathbb{L}^+-1}|\sum_{|i|\leq \mathbb{L}^+-1}(\tilde{\Omega}_{ij})^{-1}( [{P}^{1}]^{\rho_{i}}+\epsilon [{Q}^{1}]^{\rho_{i}})|e^{|j|^{1+\alpha}}\\
\nonumber&\leq& {\kappa_{2}}\sum_{|j|\leq \mathbb{L}^+-1}(|[{P}^{1}]|+\epsilon |[{Q}^{1}]|)e^{|j|^{1+\alpha}}\\
\nonumber&\lessdot& {\kappa_{2}}( \sum_{0\leq|j|\leq1}\epsilon + \epsilon\sum_{1\leq |j|\leq \mathbb{L}^+-1}e^{|j|^{1+\alpha}})\\
\nonumber &\lessdot& {\kappa_{2}}(\epsilon
+\epsilon^{\frac{4-\beta}{5}}\sum_{1\leq|j|\leq\mathbb{L}^{+}-1}e^{-(\mathbb{L}^{+})^{1+\alpha}+|j|^{1+\alpha}})\\
\nonumber &\lessdot& {\kappa_{2}}(\epsilon+ 2\mathbb{L}^{+}\epsilon^{\frac{4-\beta}{5}})\\
\nonumber &\lessdot& \epsilon^{\frac{4}{5}-\gamma-\frac{1}{5}\beta}.
\end{eqnarray}
Hence,\ on $ \mathcal{D}_{s,r-6\sigma}$, one has
\begin{eqnarray}
\nonumber{\|}F_{\theta}{\|} &=& \underset{|j|\leq \mathbb{L}^+-1}{\sum}|\tilde{F}_{\theta_{j}}+a_{j}|e^{|j|^{1+\alpha}}\\
\nonumber&\leq& \underset{|j|\leq \mathbb{L}^+-1}{\sum}|\tilde{F}_{\theta_{j}}|e^{|j|^{1+\alpha}}+ \|a\|\\
\nonumber&\leq& \underset{|j|\leq \mathbb{L}^+-1}{\sum}|\tilde{F}_{\theta_{j}}|e^{(\mathbb{L}^+-1)^{1+\alpha}}+ \|a\|\\
\nonumber&\lessdot& \epsilon^{-\frac{1+\beta}{5}}\left(\sigma^{-1}|\tilde{F}|_{s,r-5\sigma}\right)+ \epsilon^{\frac{4}{5}-\gamma-\frac{1}{5}\beta}\\
\nonumber& \lessdot & \frac{1}{\sigma^{6\mathbb{L}^+}}\epsilon^{\frac{4}{5}-3\gamma-\frac{1}{5}\beta}.
\end{eqnarray}
Similarly,\ on $ \mathcal{D}_{s-\tilde{\sigma},r-5\sigma}$,\ we obtain the estimate
\begin{eqnarray}
\nonumber|F_{\rho}|_{\infty} &=&\sup_j|F_{\rho_{j}}|\\
\nonumber&\leq& \sup_j \tilde{\sigma}^{-1}e^{|j|^{1+\alpha}}|\tilde{F}|_{s,r-5\sigma}\\
\nonumber&\lessdot& \epsilon^{-\frac{1+\beta}{5}}\left(\frac{1}{\tilde{\sigma}\sigma^{6\mathbb{L}^+-1}}\epsilon^{1-3\gamma}\right)\\
\nonumber& \lessdot & \frac{1}{\tilde{\sigma}\sigma^{6\mathbb{L}^+-1}}\epsilon^{\frac{4}{5}-3\gamma-\frac{1}{5}\beta}.
\end{eqnarray}
Recalling the estimates $ F_{\rho},F_\theta$,\ we thus have
\[
\sigma^{-1}| F_{\rho} |_{\infty},\tilde{\sigma}^{-1}{\|}F_{\theta}{\|}\lessdot  \frac{1}{\tilde{\sigma}\sigma^{6\mathbb{L}^+}}\epsilon^{\frac{4}{5}-3\gamma-\frac{1}{5}\beta}
\]
uniformly on $ \mathcal{D}_{s-\tilde{\sigma},r-6\sigma}$.\\
Since $|(I,\theta)|_{\mathcal{P}}= \max(\|I\|,|\theta|_{\infty})$,\ note $ W =
diag(\tilde{\sigma}^{-1}I_{\Lambda},{\sigma}^{-1}I_{\Lambda})$,\ and then the above estimates are equivalent to
\[
|WX_{F}|_{\mathcal{P}}\lessdot \frac{1}{\tilde{\sigma}\sigma^{6\mathbb{L}^+}}\epsilon^{\frac{4}{5}-3\gamma-\frac{1}{5}\beta}
\]
on $ \mathcal{D}_{s-\tilde{\sigma},r-6\sigma}$.

Considering the Hamiltonian vector-field $X_{F}$ associated with $ F$,\ the time-1-map can be written as
\[
 \Psi:  \mathcal{D}_{b}= D_{s-2\tilde{\sigma},r-7\sigma} \rightarrow \mathcal{D}_{a}= D_{s-\tilde{\sigma},r-6\sigma},
\]
and the estimate
\begin{eqnarray}
\label{050002*}|W(\Psi-id)|_{\mathcal{P},\mathcal{D}_{b}} \lessdot \frac{1}{\tilde{\sigma}\sigma^{6\mathbb{L}^+}}\epsilon^{\frac{4}{5}-3\gamma-\frac{1}{5}\beta}
\end{eqnarray}
holds,\ where $ \mathcal{D}_{b}$ and $\mathcal{D}_{a}$ are the domain of $ |W\cdot|_{\mathcal{P}}$-distance.

\section{The new hamiltonian}
In views of $ (\ref{N1**}) $ and $ (\ref{013})$,\ we obtain the new Hamiltonian
\begin{eqnarray}
\label{N9} H_{+} = N_{+} + P_{+} + \check{V}+ \epsilon\sum_{\lfloor i,j\rceil|dist(\lfloor i,j\rceil,0)\geq\mathbb{L}^{+}}f_{\lfloor i,j\rceil}(I^{0}_{i}+\rho_{i},I^{0}_{j}+\rho_{j},\theta_{i},\theta_{j}),
\end{eqnarray}
where $ N_{+}$ and $ P_{+}$ are given in $ (\ref{N5}) $ and $ (\ref{N6}) $ respectively.

\subsection{The new normal form $ N_{+}$}From $ (\ref{N5})$,\ $ N_{+}$ is given by
\begin{eqnarray}
\nonumber N_+ &=& \hat{e} + {N} + \langle[P^{2}]\rho,\rho\rangle+\epsilon \langle[Q^{2}]\rho,\rho\rangle +\langle[U]\rho,\rho\rangle\\
\nonumber    &=& {e}_+ +\sum_{j\in\mathbb{Z}}\omega_{j}\rho_{j} + \frac{1}{2}\sum_{j\in\mathbb{Z},|i-j|\leq 1 }\Omega_{ij}\rho_{i}\rho_{j}+ \sum_{|i|,|j|\leq \mathbb{L}^+-1}\hat{\Omega}_{ij}\rho_{i}\rho_{j},\\
\nonumber    &=& e_+ +\langle \omega(+),\rho(+)\rangle+ \frac{1}{2}\langle \Omega^+\rho(+),\rho(+)\rangle+\sum_{|j|\geq \mathbb{L}^+}\omega_{j}\rho_{j}+ \frac{1}{2}\sum_{|i|\ \mbox{or}\ |j|\geq \mathbb{L}^+}\Omega_{ij}\rho_{i}\rho_{j},
\end{eqnarray}
where $ \hat{\Omega}_{ij}= [P^{2}]^{\rho_{i}\rho_{j}}+\epsilon [Q^{2}]^{\rho_{i}\rho_{j}}+ [U]^{\rho_{i}\rho_{j}}$ and $ \Omega_{ij}^+ = \Omega_{ij}(+) + \hat{\Omega}_{ij}$ for $ |i|,|j|\leq \mathbb{L}^+-1$.\\
It follows from (\ref{036}) that
\[
|\hat{\Omega}|\leq |[P^{2}]|+\epsilon |[Q^{2}]|+ |[U]|\lessdot \epsilon + \frac{1}{\sigma^{2\mathbb{L}^+}}\epsilon^{1-\gamma}\lessdot \epsilon^{\frac{2}{3}}.
\]
Consequently,\ since $ \Omega^{+} = \Omega(+) + \hat{\Omega}$,\ the matrix ${\Omega}^{+}$ satisfies
\begin{eqnarray}
\nonumber|\Omega^{+}| &=& |\Omega(+) + \hat{\Omega}|\\
\nonumber&=& |\Omega(+)||E + \Omega^{-1}(+)\hat{\Omega}|\\
\nonumber&\lessdot& \kappa_{1}(1+2\kappa_{2}\epsilon^{\frac{2}{3}}),
\end{eqnarray}
and its inverse $ ({\Omega}^{+})^{-1} $ satisfy
\begin{eqnarray}
\nonumber |(\Omega^{+})^{-1} |
&=& |(\Omega(+) + \hat{\Omega})^{-1}| \\
\nonumber&=& |(E+\Omega^{-1}(+)\hat{\Omega})^{-1}\Omega^{-1}(+)| \\
\nonumber&\leq& |(E+\Omega^{-1}(+)\hat{\Omega})^{-1}||\Omega^{-1}(+)|\\
\nonumber&\lessdot& \frac{{\kappa_{2}}}{1-{2\kappa_{2}}\epsilon^{\frac{2}{3}}}.
\end{eqnarray}
Moreover,\ let $$ \bar{\Omega}^{+}= \begin{pmatrix}
\Omega^{+}&0\\
0&0
\end{pmatrix}_{\infty\times\infty},
$$
and then the operator $ \bar{\Omega}^{+} : \mathbb{C}^{\mathbb{Z}}\rightarrow \mathbb{C}^{\mathbb{Z}} $ satisfies
\begin{eqnarray}
\nonumber|||\bar{\Omega}^{+}||| &=& \sup_{I,\|I\|\neq0}\frac{\|\bar{\Omega}^{+} I\|}{\|I\|}\leq \sup_{I,\|I\|\neq0}\frac{\|\Omega(+)I\|+ \|\hat{\Omega} I\|}{\|I\|}\\
\nonumber&\leq& \kappa_{1} + \sup_{I,\|I\|\neq0}\frac{\underset{|i|,|j|\leq L^+-1}{\sum}|\hat{\Omega}_{ij}||I_{i}|e^{|j|^{1+\alpha}}}{\|I\|},\\
\nonumber&\lessdot& \kappa_{1} + \sup_{I,\|I\|\neq0}\frac{\left(\underset{|j|\leq L^+-1}{\sum}\epsilon^{\frac{2}{3}-\frac{1+\beta}{5}}\right)\left(\underset{|i|\leq L^+-1}{\sum}|I_{i}|e^{|i|^{1+\alpha}}\right)}{\|I\|}\\
\nonumber&\lessdot& \kappa_{1}+\epsilon^{\frac{1}{3}}.
\end{eqnarray}
Correspondingly,\ let
$$ (\bar{\Omega}^{+})^{-1}= \begin{pmatrix}
(\Omega^{+})^{-1}&0\\
0&0
\end{pmatrix}_{\infty\times\infty},
$$
and the related operator $ (\bar{\Omega}^{+} )^{-1}: \mathbb{C}^{\mathbb{Z}}\rightarrow \mathbb{C}^{\mathbb{Z}} $ satisfies
\begin{eqnarray}
\nonumber|||(\bar{\Omega}^{+})^{-1}|||&=& \sup_{I,\|I\|\neq0}\frac{\|(\bar{\Omega}^{+})^{-1} I\|}{\|I\|}
=\sup_{I,\|I\|\neq0}\frac{\|{\Omega}^{+} I(+)\|}{\|I\|}\\
\nonumber&=& \sup_{I,\|I\|\neq0}\frac{\|(E+\Omega^{-1}(+)\hat{\Omega})^{-1}\Omega^{-1}(+) I(+)\|}{\|I\|}\\
\nonumber&\leq& \sup_{I,\|I\|\neq0}\frac{\|(E+\Omega^{-1}(+)\hat{\Omega})^{-1}b(+)\|}{\|I\|}\\
\nonumber&(& \mbox{by letting}\ b(+) = \Omega^{-1}(+) I(+))\\
\nonumber&\leq& \sup_{I,\|I\|\neq0}\frac{\|\sum^{\infty}_{k=0}(\Omega^{-1}(+)\hat{\Omega})^{k}b(+)\|}{\|I\|}\\
\nonumber&\lessdot& \frac{1}{1-{2\kappa_{2}}\epsilon^{\frac{1}{3}}}\cdot\left(\sup_{I,\|I\|\neq0}\frac{\|b(+)\|}{\|I\|}\right)\\
\nonumber&\lessdot& \frac{\kappa_{2}}{1-{2\kappa_{2}}\epsilon^{\frac{1}{3}}}.
\end{eqnarray}

\subsection{The new perturbation $ P_{+}$ }Recall that the new term $ P_{+}$ is given by $ (\ref{N6})$,\ i.e.,
\begin{eqnarray}\nonumber
P_+ :&=&
\tilde{V}^{3}+\bar{V}^{3}+\Gamma P^{3}+\epsilon\Gamma Q^{3}+\tilde{V}^{4}+\bar{V}^{4}+\Gamma P^{4}+\epsilon\Gamma Q^{4}\\
&&\nonumber+ \{\tilde{V}^{3}+\Gamma P^{3}+\epsilon\Gamma Q^{3},F\}^{3}+\{\tilde{V}^{3}+\Gamma P^{3}+\epsilon\Gamma Q^{3},F\}^{4}\\
&&\nonumber+ \{\tilde{V}^{4}+\Gamma P^{4}+\epsilon\Gamma Q^{4},F\}^{3}+\{\tilde{V}^{4}+\Gamma P^{4}+\epsilon\Gamma Q^{4},F\}^{4}\\
&&\label{018}+ \int_{0}^1\{\Gamma P^{low}+\epsilon \Gamma Q^{low},F\}\circ X_F^{t}\ \mathrm{d}{t}\\
&&\label{019}+\int_{0}^1(1-t)\{\{{N}^{0}+ {N}^{1}+\tilde{V}+\bar{V}+\Gamma P^{high}+\epsilon\Gamma Q^{high} ,F\},F\}\circ X_F^{t}\ \mathrm{d}{t}
\\
&&\label{020}+ (\sum_{|v|\geq M^{+}}\widehat{P}_{v}+\epsilon \widehat{Q}_{v})\circ X_F^1\\
&&\label{021}+ \{{N}^{1},F\} + \{\bar{V},F\} + (\epsilon\sum_{\mbox{either}\ |i|\ \mbox{or}\ |j|\ \geq\ \mathbb{L}^+\ \mbox{but not both}}f_{\lfloor i,j\rceil}(\rho,\theta))\circ X_F^1.
\end{eqnarray}
By estimates (\ref{050001*}),(\ref{050002*}),\ one has
\begin{eqnarray}
|H\circ X^{t}_{F}|_{s-2\tilde{\sigma},r-7\sigma}\leq |H|_{s-\tilde{\sigma},r-6\sigma}.
\end{eqnarray}
Hence,\ with this assumption and Lemma \ref{N7},\ one can estimate new error term $ P_+ $ by
\begin{eqnarray}
\label{022}|(\ref{018})|_{s-2\tilde{\sigma},r-7\sigma}&\leq & \left|\int_{0}^1\{\Gamma P^{low}+\epsilon \Gamma Q^{low},F\}\circ X_F^{t}\ \mathrm{d}{t}\right|_{s-2\tilde{\sigma},r-7\sigma}\\
\nonumber &\leq & \int_{0}^1 |\{\Gamma P^{low}+\epsilon \Gamma Q^{low},F\} |_{s-\tilde{\sigma},r-6\sigma}\mathrm{d}{t}\\
\nonumber &\lessdot & \epsilon^{-\frac{1+\beta}{5}}\tilde{\sigma}^{-1}(\sigma^{-1}|\tilde{F}|_{s,r-5\sigma} +|a|)|\Gamma P^{low}+\epsilon \Gamma Q^{low} |_{s,r}\\
\nonumber &\lessdot & \frac{1}{\tilde{\sigma}\sigma^{6\mathbb{L}^+}}\epsilon^{\frac{4}{5}-3\gamma-\frac{1}{5}\beta}|\Gamma P^{low}+\epsilon \Gamma Q^{low} |_{s,r}\\
\nonumber &\lessdot &  \frac{1}{\tilde{\sigma}\sigma^{6\mathbb{L}^+}}\epsilon^{\frac{9}{5}-3\gamma-\frac{1}{5}\beta},
\end{eqnarray}
and
\begin{eqnarray}
\label{024}|(\ref{020})|_{s-\tilde{\sigma},r-2\sigma}&\leq & \left|\left(\sum_{|v|\geq M^{+}}(\widehat{P}_{v}+\epsilon \widehat{Q}_{v})e^{\mathbf{i}\langle v,\theta\rangle}\right)\circ X_F^1\right|_{s-\tilde{\sigma},r-2\sigma }\\
\nonumber &\leq &  \left|\sum_{|v|\geq M^{+}}(\widehat{P}_{v}+\epsilon \widehat{Q}_{v})e^{\mathbf{i}\langle v,\theta\rangle}\right|_{ s,r-\sigma}\\
\nonumber &\leq & \epsilon\sum_{M\geq M^{+}} M^{2\mathbb{L}^{+}}e^{-\sigma M}\\
\nonumber  &\lessdot& \epsilon^{2},
\end{eqnarray}
by choosing $ M^{+}= \frac{2}{\sigma}|\ln \epsilon|$.

Similarly,\ we also obtain
\begin{eqnarray}
&&\label{023}|(\ref{019})|_{s-3\tilde{\sigma},r-8\sigma} \\
\nonumber&\leq & | \int_{0}^1(1-t)\{\{{N}^{0}+ {N}^{1}+\tilde{V}+\bar{V}+\Gamma P^{high}+\epsilon\Gamma Q^{high},F\},F\}\circ X_F^{t}\ \mathrm{d}{t}
|_{s-3\tilde{\sigma},r-8\sigma}\\
\nonumber &\lessdot & | \{\{{N}^{0}+ {N}^{1}+\tilde{V}+\bar{V}+\Gamma P^{high}+\epsilon\Gamma Q^{high},F\},F\}|_{s-2\tilde{\sigma},r-7\sigma}\\
\nonumber &&(\mbox{in view of (\ref{022})})\\
\nonumber &\lessdot & \left(\frac{1}{\tilde{\sigma}\sigma^{6\mathbb{L}^+}}\epsilon^{\frac{4}{5}-3\gamma-\frac{1}{5}\beta}\right)\left| \{{N}^{0}+ {N}^{1}+\tilde{V}+\bar{V}+\Gamma P^{high}+\epsilon\Gamma Q^{high},F\}\right|_{s-\tilde{\sigma},r-6\sigma}\\
\nonumber &\lessdot &  \left(\frac{1}{\tilde{\sigma}\sigma^{6\mathbb{L}^+}}\epsilon^{\frac{4}{5}-3\gamma-\frac{1}{5}\beta}\right)^2\left|{N} ^{0}+ {N}^{1}+\tilde{V}+\bar{V}+\Gamma P^{high}+\epsilon\Gamma Q^{high}\right|_{s,r}\\
\nonumber &\lessdot &  \left(\frac{1}{\tilde{\sigma}\sigma^{6\mathbb{L}^+}}\epsilon^{\frac{4}{5}-3\gamma-\frac{1}{5}\beta}\right)^2,
\end{eqnarray}
and
\begin{eqnarray}
\label{025}&&|(\ref{021})|_{s-\tilde{\sigma},r-2\sigma}\\
\nonumber &\leq &\left|\{{N}^{1},F\}\right|_{s-\tilde{\sigma},r-2\sigma}+ \left|\{\bar{V},F\}\right|_{s-\tilde{\sigma},r-2\sigma}\\
\nonumber&& +\left|\left(\epsilon\sum_{\mbox{either}\ |i|\ \mbox{or}\ |j|\ \geq\ \mathbb{L}^+\ \mbox{but not both}} f_{\lfloor i,j\rceil}(I^0+\rho,\theta)\right)\circ X_F^1\right|_{s-\tilde{\sigma},r-2\sigma}\\
\nonumber &\leq & 2e^{-(\mathbb{L}^{+})^{1+\alpha}}
\left(\frac{1}{\sigma^{2\mathbb{L}^+}}\epsilon^{1-3\gamma}\right) +\left|\epsilon\sum_{\mbox{either}\ |i|\ \mbox{or}\ |j|\ \geq\ \mathbb{L}^+\ \mbox{but not both}}f_{\lfloor i,j\rceil}(\rho,  \theta)\right|_{s,r-\sigma}\\
\nonumber &\leq & 2
\left(\frac{1}{\sigma^{2\mathbb{L}^+}}\epsilon^{\frac{6}{5}-3\gamma+\frac{1}{5}\beta}\right) +\epsilon\exp[{-5(\mathbb{L}^{+}-1)^{1+\alpha}}]\\
\nonumber &\lessdot &  \epsilon^{1+\beta},
\end{eqnarray}
by choosing $ \mathbb{L}^{+}= (\frac{1+\beta}{5}|\ln \epsilon|)^{\frac{1}{1+\alpha}}$.

From the above estimate (\ref{022}),(\ref{023}),(\ref{024}) and (\ref{025}) contributing to $  P^{low}_{+}$,\ one easily gets that
\begin{eqnarray}
\label{026}|P^{low}_{+}|_{s-3\tilde{\sigma},r-8\sigma}&\lessdot & \left(\frac{1}{\tilde{\sigma}\sigma^{6\mathbb{L}^+}}\epsilon^{\frac{4}{5}-3\gamma-\frac{1}{5}\beta}\right)^2 + \epsilon^{1+\beta}\lessdot \epsilon^{1+\beta},
\end{eqnarray}
if $ 0< \gamma< \frac{1}{301}$ and $ 0 < \beta < \frac{1}{10}$.\\
Similarly,\ by estimates (\ref{022}),(\ref{023}),(\ref{024}) and (\ref{025}),\ we obtain
\begin{eqnarray}
\label{027}\left|\sum^{4}_{j\geq3} P_{+}^{j}\right|_{s-3\tilde{\sigma},r-8\sigma}&\lessdot & 1+\frac{1}{\tilde{\sigma}\sigma^{6\mathbb{L}^+}}\epsilon^{\frac{4}{5}-3\gamma-\frac{1}{5}\beta}
+\left(\frac{1}{\tilde{\sigma}\sigma^{6\mathbb{L}^+}}\epsilon^{\frac{4}{5}-3\gamma-\frac{1}{5}\beta}\right)^2 .
\end{eqnarray}
Thus,\ if we choose $ \epsilon_{1} = {\epsilon^{1+\beta}}$,\ (\ref{026}) will be bounded by $ \epsilon_{1}$.

\section{iteration and convergence}
Let $ 0 < \beta < \frac{1}{10}$ and $ 0 < \gamma < \frac{1}{301}$ be constants.\ In the following,\ we display the various inductive constants in the list:
 \begin{itemize}
 \item[(1)]$\epsilon_{k}=\epsilon^{({1+\beta})^{k}}$:\ $\epsilon_{k}$ bounds the size of the interaction after $ k $ iterations;


\item[(2)] Given $ 0 < s_{0}=s \leq 1 $,\ $s_{k+1}=s_{k}- 3\tilde{\sigma}_{k} = s_{k}- \frac{(k+1)^{-2}}{2\sum^{\infty}_{j=1}j^{-2}}s_0$ for $ k= 0,1,2,...$:
$ s_k$ measures the size of the analyticity domain in the action variables after $ k $ iterations,\ and  $ 3\tilde{\sigma}_{k} $ is the amount by which the domain shrinks in the $ k $-th step;

 \item[(3)] Given $ 0 < r_{0}=r\leq 1 $,\ $r_{k+1}=r_{k}- 8\sigma_{k} = r_{k}- \frac{(k+1)^{-2}}{2\sum^{\infty}_{j=1}j^{-2}}r_0$ for $ k= 0,1,2,...$:\ $r_k$ measures the size of the analyticity domain in the angular variables after $ k $ iterations,\ and  $ 8\sigma_{k} $ is the amount by which the domain shrinks in the $ k $-th step;

 \item[(4)]$ \mathbb{L}^{k+1} =\{\frac{1+\beta}{5}|\ln\epsilon_{k}|\}^{1/1+\alpha}$, $ \mathbb{L}^{0} = 1$:\ $ \mathbb{L}^{k+1} $ determines the size of the region we must consider at the $ k$-th iterative step;

 \item[(5)]$ M^{k+1} = 2 (\sigma_{k})^{-1} |\ln\epsilon_k|$:\ $  M^{k+1} $ determines the number of Fourier coefficients we must consider at the $ k$-th step of the iterations;

\item[(6)]$ \mathcal{R}^{k}$ : at each stage of iterative process we are forced to exclude a small set of the resonant frequency $ \omega$,\ $ \mathcal{R}^{k}$ is a set of the resonant frequencies remaining after $ k$ iterations ($ \mathcal{R}^{0}= \mathbb{R}^{\mathbb{Z}}$);

 \item[(7)]$ W_{k} = diag(\tilde{\sigma}_{k}^{-1}I_{\Lambda},{\sigma}_{k}^{-1}I_{\Lambda})$:\ $ W_{k}$ is the corresponding weight matrices;

 \item[(8)]$ \mathcal{D}_{k} = \mathcal{D}_{s_k,r_{k}}$:\ $ \mathcal{D}_{k}$ is the $k$-th complex domain.
 \end{itemize}

\begin{lemma}
(Iterative lemma)\ There exists $ \epsilon_{0}> 0$ such that if $ \epsilon < \epsilon_{0}$,\ then,\ for every $ k\geq 0$,\ there exists a canonical transformation,\ $ \Psi^{k}$,\ which is analytic and invertible on $  \mathcal{D}_{k}$  and maps this set into $ \mathcal{D}_{0} $.\ The Hamiltonian $ H_{k} = H_{0}\circ \Psi^{k}$ has the form
\begin{eqnarray}
H_{k} = N_{k} + R_{k}
\end{eqnarray}
where
\begin{eqnarray}
\nonumber N_{k}
&=&  e_{k} + \langle \omega(k),\rho(k)\rangle+ \frac{1}{2}\langle \Omega^{k}\rho(k),\rho(k)\rangle+  \sum_{|j|\geq \mathbb{L}^{k}}\omega_{j}\rho_{j} + \frac{1}{2}\sum_{|j|\geq \mathbb{L}^{k}}\Omega_{j}\rho^{2}_{j},
\end{eqnarray}
with $\Omega^{k}=
\Omega(k)+ \sum^{k-1}_{s=0}\hat{\Omega}^{s}= \Omega(k)+\sum^{k-1}_{s=0} [P_{s}^{2}] +\epsilon [Q_{s}^{2}]+ [U_{s}] $,\\
and
\begin{eqnarray}
\nonumber
{R}_{k} &=& P_{k}(\rho,\theta) + V_{k}(\rho)
+ \epsilon\sum_{\lfloor i,j\rceil|dist(\lfloor i,j\rceil,0)\geq{\mathbb{L}^{k}}}f_{\lfloor i,j\rceil}(I^{0}_{i}+\rho_{i},I^{0}_{j}+\rho_{j},\theta_{i},\theta_{j})
\end{eqnarray}
with $P_{k} $ depending only on $(\rho_{j},\theta_{j})$ for $ |j|\leq{\mathbb{L}^{k}}$ and $ V_{k}(\rho)= \sum_{\lfloor i,j\rceil|\ |j|\geq{\mathbb{L}^{k}}} V_{\lfloor i,j\rceil}(\rho)$.\\
Rewrite $ R_{k}$ into the following form
\begin{eqnarray}
R_{k} = P_{k} + \tilde{V}_{k}(\rho)+ \bar{V}_{k}(\rho)+\check{V}_{k}(\rho) +  R_{k,1} + R_{k,2},
\end{eqnarray}
where
\begin{eqnarray}
\nonumber
\tilde{V}_{k}(\rho) &=& \sum_{\lfloor i,j\rceil|\mathbb{L}^{k}\leq|i|,|j|\leq\mathbb{L}^{k+1}-1}V_{\lfloor i,j\rceil}(\rho),\\
\nonumber
\bar{V}_{k}(\rho) &=& \sum_{\lfloor i,j\rceil|\ |j|=\mathbb{L}^{k+1}}V_{\lfloor i,j\rceil}(\rho),\\
\nonumber
\check{V}_{k}(\rho) &=& \sum_{\lfloor i,j\rceil|\ |i|,|j|\geq\mathbb{L}^{k+1}}V_{\lfloor i,j\rceil}(\rho),
\end{eqnarray}
and
\begin{eqnarray}
\nonumber R_{k,1}&= & \epsilon\sum_{\lfloor i,j\rceil|\mathbb{L}^{k}\leq dist(\lfloor i,j\rceil,0)\leq\mathbb{L}^{k+1}-1}f_{\lfloor i,j\rceil}(I^{0}+\rho,\theta),\\
\nonumber R_{k,2}&=& \epsilon\sum_{\lfloor i,j\rceil|dist(\lfloor i,j\rceil,0)\geq\mathbb{L}^{k+1}}f_{\lfloor i,j\rceil}(I^{0}+\rho,\theta).
\end{eqnarray}
Note
\begin{eqnarray}\label{maineq1*k}
 R_{k,1} &=& \epsilon Q_{k}(\rho,\theta)+ \epsilon\sum_{\mbox{either}\ |i|\ \mbox{or}\ |j|\ \geq\ \mathbb{L}^{k+1}\ \mbox{but not both}}f_{\lfloor i,j\rceil}(\rho,\theta),
\end{eqnarray}
with
\begin{eqnarray}
\nonumber Q_{k} &=& \sum_{\lfloor i,j\rceil|\mathbb{L}^{k}\leq |i|,|j|\leq \mathbb{L}^{k+1}-1}f_{\lfloor i,j\rceil}(I^{0}_{i}+\rho_i,I^0_{j}+\rho_j,\theta_{i},\theta_{j}).
\end{eqnarray}
Suppose $ P_{k}^{low}(\rho,\theta)$ satisfies the smallness assumption
\begin{eqnarray}
\nonumber |P^{low}_{k}|_{s_{k},r_{k}}\lessdot \epsilon_{k},
\end{eqnarray}
and $ P_{k}^{high}(\rho,\theta)$ satisfies
\begin{eqnarray}
\nonumber |P^{high}_{k}|_{s_{k},r_{k}}\lessdot 1+\sum^{k-1}_{i=0}\epsilon^{\frac{1+\beta}{2}}_{i}.
\end{eqnarray}
Then for $ \omega = \omega(k+1)\in (\mathcal{A}_{1})$,\ the $ k$-th homological equation
\begin{eqnarray}\label{N2K}
\{{N}_{k},{F_{k}}\}&+&\Gamma P_{k}^{low}+\epsilon \Gamma Q_{k}^{low}+ \{\tilde{V}_{k}+\Gamma P_{k}^{high}+\epsilon \Gamma Q_{k}^{high},F_{k}\}^{low}\\
\nonumber&=& [P_{k}^{0}]+\epsilon [Q_{k}^{0}]+ [P_{k}^{1}]+ \epsilon [Q_{k}^{1}] + [P_{k}^{2}]+ \epsilon [Q_{k}^{2}] + [U_{k}],
\end{eqnarray}
with $ \langle U_{k}\rho,\rho\rangle=
\{\tilde{V}_{k}+\Gamma P_{k}^{high}+\epsilon \Gamma Q_{k}^{high},F_{k}\}^{low}$,\
has a solution $ F_{k} =\tilde{F}_{k} + \langle a^{k},\theta\rangle=  F^{0}_{k}+ F^{1}_{k} + F^{2}_{k}+\langle a^{k},\theta\rangle $ with the estimates
\begin{eqnarray}
\label{1050001}\|a^{k} \|\lessdot \epsilon^{\frac{3}{5}-\frac{1}{5}\beta-\gamma}_{k},
\end{eqnarray}
and
\begin{eqnarray}
\label{050001} | \tilde{F}_{k} |_{s_{k+1},r_{k+1}}\lessdot \epsilon_{k}^{\frac{11}{20}}.
\end{eqnarray}
Moreover,\
\begin{eqnarray}
H_{k+1} = N_{k+1} + R_{k+1},
\end{eqnarray}
where
\begin{eqnarray}
\nonumber N_{k+1} &=&  e_{k+1} + \langle\omega(k+1),\rho(k+1)\rangle + \frac{1}{2}\langle\Omega^{k+1}\rho(k+1),\rho(k+1)\rangle\\
\nonumber &&+ \sum_{|j|\geq \mathbb{L}^{k+1}}\omega_{j}\rho_{j} + \frac{1}{2}\sum_{|i|\ \mbox{or}\ |j|\geq \mathbb{L}^{k+1}}\Omega_{ij}\rho_{i}\rho_{j},
\end{eqnarray}
with $ {\Omega}^{k+1}= \Omega(k+1) + \sum^{k}_{s=0}\hat{\Omega}^{s}= \Omega(k+1) + \sum^{k}_{s=0}[P_{s}^{2}]+\epsilon [Q_{s}^{2}]+ [U_{s}]$,\\
and
\begin{eqnarray}
\nonumber
{R}_{k+1} &=& V_{k+1}(\rho)+ P_{k+1}(\rho,\theta) \\
\nonumber&&+ \epsilon\sum_{\lfloor i,j\rceil|dist(\lfloor i,j\rceil,0)\geq{\mathbb{L}^{k+1}}}f_{\lfloor i,j\rceil}(I^0_{i}+\rho_i,I^0_{j}+\rho_j,\theta_{i},\theta_{j}),
\end{eqnarray}
in which $ V_{k+1}(\rho)= \sum_{\lfloor i,j\rceil|\ |j|\geq{\mathbb{L}^{k+1}}} V_{\lfloor i,j\rceil}(\rho)$,\ with the following estimates hold:\\
(1) The symplectic map $ \Psi^{k+1} = \Psi^{k}\circ\Psi_{k}$ satisfies
\begin{eqnarray}
\label{200} |W_{0}(\Psi^{k+1}-\Psi^{k})| \lessdot \epsilon_{k}^{\frac{17}{50}}.
\end{eqnarray}
(2) The $ (2\mathbb{L}^{k+1}-1)\times (2\mathbb{L}^{k+1}-1)$ matrix $\Omega^{k+1} $,\ the relative operator $ \bar{\Omega}^{k+1}: \mathbb{C}^{\mathbb{Z}}\rightarrow \mathbb{C}^{\mathbb{Z}} $ respectively satisfy
\begin{eqnarray}
\label{201}|\Omega^{k+1}| \lessdot \kappa_{1}(1+ 2\kappa_{2}\sum^{k}_{s=0}\epsilon^{\frac{8}{15}}_{s}), |||\bar{\Omega}^{k+1} |||\lessdot \kappa_{1}+ \sum^{k}_{s=0}\epsilon^{\frac{4}{15}}_{s},
\end{eqnarray}
and its inverse matrix $(\Omega^{k+1})^{-1}$,\ inverse operator $ (\bar{\Omega}^{k+1})^{-1}: \mathbb{C}^{\mathbb{Z}}\rightarrow \mathbb{C}^{\mathbb{Z}}$ respectively satisfy
\begin{eqnarray}
\label{202}
|(\Omega^{k+1})^{-1}|\lessdot \frac{{\kappa_{2}}}{1-{2\kappa_{2}}\sum^{k}_{i=0}\epsilon^{\frac{8}{15}}_{i}}, |||(\bar{\Omega}^{k+1})^{-1}|||\lessdot \frac{{\kappa_{2}}}{1-{2\kappa_{2}}\sum^{k}_{i=0}\epsilon^{\frac{4}{15}}_{i}}.
\end{eqnarray}
(3) The perturbation $ P^{low}_{k+1} $ satisfies
\begin{eqnarray}
\label{203} |P^{low}_{k+1}|_{s_{k+1},r_{k+1}} \lessdot \epsilon_{k+1},
\end{eqnarray}
and $ P_{k+1}^{high}(\rho,\theta)$ satisfies
\begin{eqnarray}
\label{204} |P^{high}_{k+1}|_{s_{k+1},r_{k+1}} \lessdot 1+\sum^{k}_{i=0}\epsilon^{\frac{1+\beta}{2}}_{i}.
\end{eqnarray}
\end{lemma}
\begin{proof}
First of all,\ distinguishing terms of (\ref{N2K}) by the order of $ \rho $,\ we have
\begin{equation}\label{K**}
\begin{cases}
-\partial_{\omega}F_{k}^{0}-\langle \omega(k+1), a^{k}\rangle + \Gamma{P_{k}}^{0}+ \epsilon\Gamma{Q_{k}}^{0}=0,\\
-\partial_{\omega}F_{k}^{1}-\tilde{\Omega}^{k} a^{k} -\tilde{\Omega}^{k}\partial_{\theta}F_{k}^{0}+\Gamma{P_{k}}^{1}+ \epsilon\Gamma{Q_{k}}^{1}=0,\\
-\partial_{\omega}F_{k}^{2}-\tilde{\Omega}^{k}\partial_{\theta}F_{k}^{1}+ U_{k} +\Gamma{P_{k}}^{2}+ \epsilon\Gamma{Q_{k}}^{2}=0,
\end{cases}
\end{equation}
where $ \tilde{\Omega}^{k} = {\Omega}(k+1) + \sum^{k-1}_{s=0} \hat{\Omega}^{s}$.

Now we will prove the Lemma by the following steps.

$\mathbf{Step\ 1.}$ $ \mathbf{ Proof\  of\  (\ref{050001})}$.\
We consider the first equation of (\ref{K**}).\ Since $ \omega $ satisfies the Diophantine condition A(1),\ we can solve the first equation of (\ref{K**})
\begin{eqnarray}
\partial_{\omega}F_{k}^{0}= \Gamma{P}_{k}^{0}+ \epsilon\Gamma{Q}_{k}^{0}-([{P}_{k}^{0}] + \epsilon[{Q}_{k}^{0}] +\langle \omega, a^{k}\rangle ),
\end{eqnarray}
then the solution of the homological equation is given by
\begin{eqnarray}
\label{2009}F_{k}^{0}(\theta)&=&\sum_{0< |v| < M^{k+1}} \frac{\hat{P}_{k,v}^{0}+\epsilon \hat{Q}_{k,v}^{0}}{\mathbf{i}\langle \omega(k+1),v\rangle}e^{\mathbf{i}\langle v,\theta\rangle}.
\end{eqnarray}
From Diophantine condition $(\mathcal{A}_{1})$,\ we have
\begin{eqnarray}
\nonumber|F_{k}^{0}(\theta)|_{r_{k}-\sigma_{k}} &\lessdot& \epsilon_{k}^{-\gamma}\sum_{0< |v| < M^{k+1}}\left(\epsilon_{k} +\epsilon \underset{\mathbb{L}^{k}\leq|j|\leq \mathbb{L}^{k+1}-1}{\sum}|I^{0}_{|j|-1}|^{5}\right) e^{-r_{k}|v|}e^{(r_{k}-\sigma_{k})|v|}\\
\nonumber &\lessdot& \sigma_{k}^{-1}(2M^{k+1})^{2\mathbb{L}^{k+1}}\epsilon_{k}^{1-\gamma}\\
\label{3009} &\lessdot& \sigma_{k}^{-1}\epsilon_{k}^{1-2\gamma},\  0 < \sigma_{k} < r_{k}.
\end{eqnarray}
From Cauchy estimate,\ one has
\[
|\partial_{\theta}F_{k}^{0}(\theta)|_{r_{k}-2\sigma_{k}} \lessdot \frac{1}{\sigma_{k}}|F_{k}^{0}(\theta)|_{r_{k}-\sigma_{k}},
\]
which together with (\ref{3009}) implies
\begin{eqnarray}
\label{2010}|\partial_{\theta}F_{k}^{0}(\theta)|_{r_{k}-2\sigma_{k}}  &\lessdot& \frac{\epsilon_{k}^{1-2\gamma}}{\sigma_{k}^{2}}.
\end{eqnarray}
We next consider the second equation of (\ref{K**}).\\
Since $ \widehat{F}_{k,0}^{1}=0 $,\ we can choose a vector $ a^{k} $ such that $ (\tilde{\Omega}^{k})a^{k} -[{P_{k}}^{1}]-\epsilon [{Q_{k}}^{1}]=0 $.\ Then we have
\begin{equation}\label{012*}
a^{k} = (\tilde{\Omega}^{k})^{-1}([{P}_{k}^{1}] +\epsilon [{Q}_{k}^{1}]),
\end{equation}
with the estimate
\begin{eqnarray}
\nonumber |a^{k}| &=& |(\tilde{\Omega}^{k})^{-1}||[{P}_{k}^{1}] +\epsilon [{Q}_{k}^{1}]|\\
\nonumber &\leq& |{\Omega}^{-1}(k+1)| |(E+ \sum^{k-1}_{s=0} \hat{\Omega}^{s})^{-1}|(|[{P}_{k}^{1}]| +\epsilon |[{Q}_{k}^{1}]|)\\
\nonumber &\lessdot& \epsilon^{\frac{4}{5}}_{k}.
\end{eqnarray}
Let $ \tilde{F}_{k}^{0}= -\tilde{\Omega}\cdot\partial_{\theta}F_{k}^{0}$.\ We get
\[
|\tilde{F}_{k}^{0}(\theta)|_{r_{k}-2\sigma_{k}}\leq |\tilde{\Omega}||\partial_{\theta}F_{k}^{0}(\theta)|_{r_{k}-2\sigma_{k}}\lessdot \frac{1}{\sigma_{k}^{2}}\epsilon_{k}^{1-2\gamma}.
\]
Similarly,\ we can solve the second equation (\ref{K**})
\begin{eqnarray}
\label{010*}F_{k}^{1}(\theta)&=& \sum_{0< |v| < M^{k+1}}\frac{\widehat{P}_{k,v}^{1}+\epsilon \widehat{Q}_{k,v}^{1}+\widehat{\tilde{F}}_{k,v}^{0}}
{\mathbf{i}\langle \omega(k+1),v\rangle}e^{\mathbf{i}\langle v,\theta\rangle}.
\end{eqnarray}
From (\ref{2010}),\ we then obtain the estimate
\begin{eqnarray}
\nonumber&&|F_{k}^{1}(\theta)|_{r_{k}-3\sigma_{k}}\\
 \nonumber &\lessdot& \epsilon_{k}^{-\gamma}\sum_{0< |v| < M^{k+1}}\left(\epsilon^{\frac{4}{5}}_{k}+ \epsilon e^{-4(\mathbb{L}^{k}-1)^{1+\alpha}}  +\frac{\kappa_{2}}{\sigma_{k}^{2}}\epsilon_{k}^{1-2\gamma}\right) e^{-(r-2\sigma_{k})|v|}e^{(r-3\sigma_{k})|v|}\\
\nonumber&\lessdot& \epsilon_{k}^{-\gamma}\sum_{0< |v| < M^{k+1}}\left(\epsilon^{\frac{4}{5}}_{k}+\frac{\kappa_{2}}{\sigma_{k}^{2}}\epsilon_{k}^{1-2\gamma}\right) e^{-(r_{k}-2\sigma_{k})|v|}e^{(r_{k}-3\sigma_{k})|v|}\\
\nonumber&\lessdot& \frac{1}{\sigma_{k}^{3}}\epsilon_{k}^{\frac{4}{5}-4\gamma},
\end{eqnarray}
and
\begin{eqnarray}
\label{1011*}|\partial_{\theta}F_{k}^{1}(\theta)|_{r_{k}-4\sigma_{k}}\lessdot \frac{1}{\sigma_{k}^{4}}\epsilon_{k}^{\frac{4}{5}-4\gamma}.
\end{eqnarray}
Now we consider the third equation of (\ref{K**}).\\
Let $  \tilde{F}_{k}^{1}= -\tilde{\Omega}^{k}\partial_{\theta}F_{_{k}}^{1}$.\ We have
\[
|\widetilde{F}_{k}^{1}(\theta)|_{r_{k}-4\sigma_{k}}\leq |\tilde{\Omega}^{k }||\partial_{\theta}F_{k}^{1}(\theta)|_{r_{k}-4\sigma_{k}}\lessdot \frac{1}{\sigma_{k}^{4}}\epsilon_{k}^{\frac{4}{5}-4\gamma}.
\]
On $ \mathcal{D}_{s_{k}-\tilde{\sigma}_{k},r_{k}-4\sigma_{k}}$.\\
Since
\[
\langle U_{k}\rho,\rho\rangle = \{\tilde{V}_{k}^{3} + \Gamma P_{k}^{3}+\epsilon \Gamma Q_{k}^{3},F_{k}^{0}+ \langle a^{k},\theta\rangle\},
\]
and
\begin{eqnarray}
\nonumber \{\tilde{V}_{k}^{3},F_{k}^{0}+ \langle a^{k},\theta\rangle\} &=& \sum_{i,j,l}\frac{\partial^{3}\tilde{V}_{k}^{3}}{\partial \rho_{i}\partial \rho_{j}\partial \rho_{l}}(\partial_{\theta_{j}}F_{k}^{0}+ a^{k}_{j})\rho_{i}\rho_{j},
\end{eqnarray}
we easily have
\begin{eqnarray}
\nonumber |\{\tilde{V}_{k}^{3},F_{k}^{0}+ \langle a^{k},\theta\rangle\}^{\rho\rho}| &\leq& \sup_{j}\left(\left|\frac{\partial^{3}\tilde{V}_{k}^{3}}{\partial \rho_{j-1}^{2}\partial \rho_{j}}\right|+\left|\frac{\partial^{3}\tilde{V}_{k}^{3}}{\partial \rho_{j}^{3}}\right|+
\left|\frac{\partial^{3}\tilde{V}_{k}^{3}}{\partial \rho_{j+1}^{2}\partial \rho_{j}}\right|\right)
(|\partial_{\theta}F_{k}^{0}|+ |a^{k}|),\\
\nonumber &\lessdot & |\partial_{\theta}F_{k}^{0}|+ |a^{k}|,
\end{eqnarray}
and
\begin{eqnarray}
\nonumber &&|\{\Gamma P_{k}^{3}+\epsilon \Gamma Q_{k}^{3},F_{k}^{0}+ \langle a^{k},\theta\rangle\}^{\rho\rho}|\\
\nonumber &=& \max_i\sum_{|j|\leq \mathbb{L}^{k+1}-1}\left|\sum_{l}\frac{\partial^{3}(\Gamma P_{k}^{3}+\epsilon \Gamma Q_{k}^{3})}{\partial\rho_{i}\partial\rho_{j}\partial \rho_{l}}(\partial_{\theta_{l}}F_{k}^{0}+ a^{k}_{l})\right|\\
\nonumber &\lessdot& \left(\sum_{|j|\leq\mathbb{L}^{k}} \epsilon^{\frac{2}{5}}_{k}+\sum_{ \mathbb{L}^{k}\leq|j|\leq \mathbb{L}^{k+1}-1}\epsilon|I^{0}_{|j|-1}|^{2}\right)(|\partial_{\theta}F_{k}^{0}|+ |a^{k}|)\\
\nonumber &\lessdot& |\partial_{\theta}F_{k}^{0}|+ |a^{k}|.
\end{eqnarray}
One then obtains
\begin{eqnarray}
\label{037}|U_{k}|_{r_{k}-2\sigma_{k}} &\lessdot& \frac{1}{\sigma_{k}^{2}}\epsilon_{k}^{\frac{4}{5}-2\gamma}.
\end{eqnarray}
Therefore,\ we can solve the third equation of (\ref{K**})
\begin{eqnarray}
\nonumber F_{k}^{2}(\theta)&=& \sum_{0< |v| < M^{k+1}}\frac{\widehat{P}_{k,v}^{2}+\epsilon \widehat{Q}_{k,v}^{2}+\widehat{\tilde{F}}_{k,v}^{1}+\widehat{U}_{k,v}}{\mathbf{i}\langle {\omega(k+1),v\rangle}}e^{\mathbf{i}\langle v,\theta\rangle},
\end{eqnarray}
and we also obtain
\begin{eqnarray}
\nonumber&&|F_{k}^{2}(\theta)|_{r_{k}-5\sigma_{k}}\\
 \nonumber &\lessdot& \epsilon_{k}^{-\gamma}\sum_{0< |v| < M^{k+1}}\left(\sum_{|j|\leq \mathbb{L}^{k}}\epsilon^{\frac{3}{5}}_{k}+ \epsilon e^{-3(\mathbb{L}^{k}-1)^{1+\alpha}}  +\frac{\kappa_{2}}{\sigma_{k}^{4}}\epsilon^{\frac{4}{5}-4\gamma}\right) e^{-(r-2\sigma)|v|}e^{(r-3\sigma)|v|}\\
\nonumber&\lessdot& \epsilon_{k}^{-\gamma}\sum_{0< |v| < M^{k+1}}\left(\epsilon^{\frac{3}{5}-\gamma}_{k}+\frac{\kappa_{2}}{\sigma_{k}^{4}}
\epsilon_{k}^{\frac{4}{5}-4\gamma}\right) e^{-(r_{k}-4\sigma_{k})|v|}e^{(r_{k}-5\sigma_{k})|v|}\\
\nonumber&\lessdot& \frac{1}{\sigma_{k}^{5}}\epsilon_{k}^{\frac{3}{5}-6\gamma}.
\end{eqnarray}
Consequently,\ one obtains
\begin{eqnarray}
\nonumber &&|\tilde{F}_{k}|_{s^{k},r_{k}-5\sigma_{k}}\\
 \nonumber&\leq& |F_{k}^{0}|_{r_{k}-5\sigma_{k}}+|F_{k}^{1}|_{r_{k}-5\sigma_{k}}\left(\underset{|j|\leq \mathbb{L}^{k+1}-1}{\sum}|\rho_{j}|\right)+|F^{2}|_{r_{k}-5\sigma_{k}}\left(\underset{|i|,|j|\leq \mathbb{L}^{k+1}-1}{\sum}|\rho_{i}||\rho_{j}|\right)\\
\nonumber&\lessdot& \sigma_{k}^{-1}\epsilon_{k}^{1-2\gamma}+ \frac{1}{\sigma_{k}^{3}}\epsilon_{k}^{\frac{4}{5}-4\gamma}\left(\underset{|j|\leq \mathbb{L}^{k+1}-1}{\sum}e^{-|j|^{1+\alpha}}\right)+ \frac{1}{\sigma_{k}^{5}}\epsilon_{k}^{\frac{3}{5}-6\gamma}\left(\underset{|j|\leq \mathbb{L}^{k+1}-1}{\sum}e^{-|j|^{1+\alpha}}\right)^{2}\\
\nonumber&\lessdot& \frac{1}{\sigma_{k}^{5}}\epsilon_{k}^{\frac{3}{5}-6\gamma}\\
\nonumber&\lessdot& \epsilon_{k}^{\frac{11}{20}}.
\end{eqnarray}

$\mathbf{Step\ 2.}$ $ \mathbf{ Proof\  of\  (\ref{200})}$.\
From the equation of motion
\begin{eqnarray}
\dot{\theta}= F_{k,\rho}(\rho,\theta), \dot{\rho}= F_{k,\theta}(\rho,\theta)= \tilde{F}_{k,\theta}(\rho,\theta)+ a^{k},
\end{eqnarray}
vector $ a^{k} $ has to belong to $ \mathbb{C}^{\mathbb{Z}}$.\\
Recall that
\[
(\tilde{\Omega}^{k})^{-1}= (\Omega(k+1)+ \sum^{k-1}_{s=0}\hat{\Omega}^{s})^{-1} = (E+({\Omega}^{k+1})^{-1} \sum^{k-1}_{s=0}\hat{\Omega}^{s})^{-1}\Omega^{-1}(k+1),
\]
Thus,\ one has
\[
 |(\tilde{\Omega}^{k})^{-1}| \leq 4\kappa_{2}
\]
and
\begin{eqnarray}
\label{1012*}\| a^{k}\|&=& \sum_{|j|\leq \mathbb{L}^{k+1}-1}|\sum_{|i|\leq \mathbb{L}^{k+1}-1}({\Omega}^{k}_{ij})^{-1}( [{P_{k}}^{1}]^{\rho_{i}}+\epsilon [{Q_{k}}^{1}]^{\rho_{i}})|e^{|j|^{1+\alpha}}\\
\nonumber &\leq&\sum_{|i|,|j|\leq \mathbb{L}^{k}}|(\tilde{\Omega}^{k}_{ij})^{-1}[{P_{k}}^{1}]^{\rho_{i}}|e^{|j|^{1+\alpha}}\\
\nonumber&&+
\epsilon\sum_{\mathbb{L}^{k}\leq |i|,|j|\leq \mathbb{L}^{k+1}-1}|(\tilde{\Omega}^{k}_{ij})^{-T}[{Q_{k}}^{1}]^{\rho_{i}}|e^{|j|^{1+\alpha}}\\
\nonumber &\leq& A + B,
\end{eqnarray}
where
\begin{eqnarray}
\nonumber A &=&\sum_{|i|,|j|\leq \mathbb{L}^{k}}|(\tilde{\Omega}^{k}_{ij})^{-1}[{P_{k}}^{1}]^{\rho_{i}}|e^{|j|^{1+\alpha}},\\
\nonumber B &=&\epsilon\sum_{\mathbb{L}^{k}\leq |i|,|j|\leq \mathbb{L}^{k+1}-1}|(\tilde{\Omega}^{k}_{ij})^{-1}[{Q_{k}}^{1}]^{\rho_{i}}|e^{|j|^{1+\alpha}}.
\end{eqnarray}
By some simple calculations,\ one gets
\begin{eqnarray}
\nonumber A &\leq& {4}{\kappa_{2}}\sum_{|j|\leq \mathbb{L}^{k}}(\epsilon^{-\frac{1+\beta}{5}}_{k-1}\epsilon_{k })e^{|j|^{1+\alpha}}\\
\nonumber&\leq& 4{\kappa_{2}}\epsilon^{\frac{3(1+\beta)}{5}}_{k-1}\left(\sum_{|j|\leq \mathbb{L}^{k}}e^{-(\mathbb{L}^{k})^{1+\alpha}}e^{|j|^{1+\alpha}}\right)\\
\nonumber &\leq& 4{\kappa_{2}}\left(2\mathbb{L}^{k}\epsilon^{\frac{3(1+\beta)}{5}}_{k-1}\right)\\
\nonumber &\lessdot& \epsilon^{\frac{3}{5}-\gamma}_{k},
\end{eqnarray}
and
\begin{eqnarray}
\nonumber B &\leq& {4\kappa_{2}\epsilon}\sum_{\mathbb{L}^{k}\leq |j|\leq \mathbb{L}^{k+1}-1}|I^{0}_{\mathbb{L}^{k}-1}|^{4}e^{|j|^{1+\alpha}}\\
\nonumber &\leq&{4\kappa_{2}\epsilon}e^{-4(\mathbb{L}^{k}-1)^{1+\alpha}+(\mathbb{L}^{k+1}-1)^{1+\alpha} }\sum_{\mathbb{L}^{k}\leq |j|\leq \mathbb{L}^{k+1}-1}e^{-(\mathbb{L}^{k+1}-1)^{1+\alpha}+|j|^{1+\alpha}}\\
\nonumber&\leq& {4\kappa_{2}\epsilon} \epsilon^{\frac{3-\beta}{5}}_{k}\cdot (2\mathbb{L}^{k+1})\\
\nonumber &\lessdot& \epsilon^{\frac{3}{5}-\frac{1}{5}\beta-\gamma}_{k}.
\end{eqnarray}
Therefore,\ we obtain
\[
\| a^{k}\| \lessdot \epsilon^{\frac{3}{5}-\frac{1}{5}\beta-\gamma}_{k}.
\]
On $ \mathcal{D}_{s_{k},r_{k}-6\sigma_{k}}$, one also has
\begin{eqnarray}
\nonumber{\|}F_{k,\theta}{\|} &=& \underset{|j|\leq \mathbb{L}^{k+1}-1}{\sum}|\tilde{F}_{k,\theta_{j}}+a^{k}_{j}|e^{|j|^{1+\alpha}}\\
\nonumber&\leq& \underset{|j|\leq \mathbb{L}^{k+1}-1}{\sum}|\tilde{F}_{k,\theta_{j}}|e^{(\mathbb{L}^{k+1}-1)^{1+\alpha}}+ \|a^{k}\|\\
\nonumber&\lessdot& \epsilon_{k}^{-\frac{1+\beta}{5}}\sigma_{k}^{-1}|\tilde{F}_{k}|_{s_{k},r_{k}-5\sigma_{k}}+ \epsilon^{\frac{3}{5}-\frac{1}{5}\beta-\gamma}_{k}\\
\nonumber& \lessdot & \sigma_{k}^{-6}\epsilon^{\frac{2}{5}-6\gamma-\frac{1}{5}\beta}.
\end{eqnarray}
On $ \mathcal{D}_{s_{k}-\tilde{\sigma}_{k},r_{k}-5\sigma_{k}}$,\ we obtain the estimate
\begin{eqnarray}
\nonumber|F_{k,\rho}|_{\infty} &=&\sup_j|F_{k,\rho_{j}}|\\
\nonumber&\leq& \sup_j \tilde{\sigma}_{k}^{-1}e^{|j|^{1+\alpha}}|\tilde{F}_{k}|_{s_{k},r_{k}-5\sigma_{k}}\\
\nonumber&\lessdot& \epsilon_{k}^{-\frac{1+\beta}{5}}\tilde{\sigma}_{k}^{-1}\sigma_{k}^{-5}\epsilon_{k}^{\frac{3}{5}-6\gamma}\\
\nonumber& \lessdot& \tilde{\sigma}_{k}^{-1}\sigma_{k}^{-5}\epsilon_{k}^{\frac{2}{5}-6\gamma-\frac{1}{5}\beta}.
\end{eqnarray}
Recalling the estimates for $ F_{k,\rho},F_{k,\theta}$,\ we thus have
\[
\sigma_{k}^{-1}| F_{k,\rho} |_{\infty},\tilde{\sigma}_{k}^{-1}{\|}F_{k,\theta}{\|}\lessdot  \tilde{\sigma}_{k}^{-1}\sigma_{k}^{-6}\epsilon_{k}^{\frac{2}{5}-6\gamma-\frac{1}{5}\beta}
\]
uniformly on $ \mathcal{D}_{s_{k}-\tilde{\sigma}_{k},r_{k}-6\sigma_{k}}$.\\
Noting $ W_{k} =
diag(\tilde{\sigma}_{k}^{-1}I_{\Lambda},{\sigma}_{k}^{-1}I_{\Lambda})$,\ the above estimates are equivalent to
\[
|W_{k}X_{F_{k}}|_{\mathcal{P}}\lessdot \tilde{\sigma}_{k}^{-1}\sigma_{k}^{-6}\epsilon_{k}^{\frac{2}{5}-6\gamma-\frac{1}{5}\beta}
\]
on $ \mathcal{D}_{s_{k}-\tilde{\sigma}_{k},r_{k}-6\sigma_{k}}$.

Considering the Hamiltonian vector-field $X_{F_{k}} = \Psi_{k}$ associated with $ F_{k}$,\ the time-1-map can be written as
\[
\Psi_{k}:  \mathcal{D}_{s_{k}-2\tilde{\sigma}_{k},r_{k}-7\sigma_{k}} \rightarrow \mathcal{D}_{s_{k}-\tilde{\sigma}_{k},r_{k}-6\sigma_{k}},
\]
for which the estimate
\begin{eqnarray}
\label{1050002*}|W_{k}(\Psi_{k}-id)|_{\mathcal{P},\mathcal{D}_{k+1}} \lessdot \tilde{\sigma}_{k}^{-1}\sigma_{k}^{-6}\epsilon_{k}^{\frac{2}{5}-6\gamma-\frac{1}{5}\beta}
\end{eqnarray}
holds.\\
Since $ \Psi^{k+1} = \Psi^{k}\circ \Psi_{k} $ ,\ write
\begin{eqnarray}
|W_{0}(\Psi^{k+1} -\Psi^{k})|_{\mathcal{P},\mathcal{D}_{k+1}} &=& |W_{0}(\Psi^{k}\circ \Psi_{k} -\Psi^{k})|_{\mathcal{P},\mathcal{D}_{k+1}}\\
\nonumber&\leq& |W_{0}D\Psi^{k}W^{-1}_{k}|_{\mathcal{P},\mathcal{D}_{k}}|W_{k}(\Psi_{k}-id)|_{\mathcal{P},\mathcal{D}_{k+1}}.
\end{eqnarray}
By the inductive construction,\ $ \Psi^{k} = \Psi_{0}\circ\cdot\cdot\cdot\circ\Psi_{k-1}$,\ and
\begin{eqnarray}
|W_{v}D\Psi_{v}W^{-1}_{v}|_{\mathcal{P},\mathcal{D}_{v+1}}
\lessdot 1+ \epsilon_{v}^{\frac{8}{25}},\  (\mbox{in view of}\ (\ref{1050002*}))
\end{eqnarray}
thus,\ we obtain
\begin{eqnarray}
|W_{0}(\Psi^{k+1} -\Psi^{k})|_{\mathcal{P},\mathcal{D}_{k+1}} &\leq& \prod^{k-1}_{v=0}|W_{v}D\Psi_{v}W^{-1}_{v+1}|_{\mathcal{P},\mathcal{D}_{k}}
|W_{k}(\Psi_{k}-id)|_{\mathcal{P},\mathcal{D}_{k+1}}\\
\nonumber&\leq& \prod^{k-1}_{v=0}(1+ \epsilon_{v}^{\frac{8}{25}})\sigma_{k}^{-6}\epsilon_{k}^{\frac{2}{5}-6\gamma-\frac{1}{5}\beta}\\
\nonumber&\lessdot& \epsilon_{k}^{\frac{17}{50}}.
\end{eqnarray}

$\mathbf{Step\  3.}$ $ \mathbf{ Proofs\  of\  (\ref{201})\ and\ (\ref{202})}$.\ On $ \mathcal{D}_{s_{k}-\tilde{\sigma}_{k},r_{k}-2\sigma_{k}}$,\ it follows from (\ref{037}) that
\[
|\hat{\Omega}^{k}|\lessdot \epsilon^{\frac{3}{5}-\gamma}_{k} +\sigma_{k}^{-2}\epsilon_{k}^{\frac{4}{5}-2\gamma}\lessdot \epsilon_{k}^{\frac{8}{15}}.
\]
Therefore,\ since $ \Omega^{k+1} = \Omega(k+1) + \sum^{k}_{s=0}\hat{\Omega}^{s}$ for $ \hat{\Omega}^{s}= ( \hat{\Omega}_{ij})$ of $ |i|,|j|\leq \mathbb{L}^{s+1}-1$,\ the matrix ${\Omega}^{k+1}$ satisfies
\begin{eqnarray}
\nonumber|\Omega^{k+1}| &=& \left|\Omega(k+1) + \sum^{k}_{s=0}\hat{\Omega}^{s}\right|\\
\nonumber&=& |\Omega(k+1)|\left|E + \Omega^{-1}(k+1)\sum^{k}_{s=0}\hat{\Omega}^{s}\right|\\
\nonumber&\lessdot& {\kappa_{1}}({1+2{\kappa_{2}}\sum^{k}_{s=0}\epsilon^{\frac{8}{15}}_{s}}),
\end{eqnarray}
and its inverse $ ({\Omega}^{k+1})^{-1} $ satisfies
\begin{eqnarray}
\nonumber|(\Omega^{k+1})^{-1} |
&=& \left|\left(\Omega(k+1) + \sum^{k}_{s=0}\hat{\Omega}^{s}\right)^{-1}\right| \\
\nonumber&=& \left|\left(E+\Omega^{-1}(k+1)\sum^{k}_{s=0}\hat{\Omega}^{s}\right)^{-1}\Omega^{-1}(k+1)\right| \\
\nonumber&\lessdot& \frac{\kappa_{2}}{1-{2\kappa_{2}}\sum^{k}_{s=0}\epsilon_{s}^{\frac{8}{15}}}.
\end{eqnarray}
Moreover,\ since $$ \bar{\Omega}^{k+1}= \begin{pmatrix}
\Omega^{k+1}&0\\
0&0
\end{pmatrix}_{\infty\times\infty},
$$
the operator $ \bar{\Omega}^{k+1} : \mathbb{C}^{\mathbb{Z}}\rightarrow \mathbb{C}^{\mathbb{Z}} $ satisfies
\begin{eqnarray}
\nonumber|||\bar{\Omega}^{k+1}||| &=& \sup_{I,\|I\|\neq0}\frac{\|\bar{\Omega}^{k+1} I\|}{\|I\|}\leq \sup_{I,\|I\|\neq0}\frac{\|\Omega(k+1)I\|+ \|\sum^{k}_{s=0}\hat{\Omega}^{s} I\|}{\|I\|}\\
\nonumber&\leq& \kappa_{1} + \sup_{I,\|I\|\neq0}\sum^{k}_{s=0}\left(\frac{\sum_{|i|,|j|\leq \mathbb{L}^{s+1}-1}|\hat{\Omega}^{s}_{ij}||I_{i}|e^{|j|^{1+\alpha}}}{\|I\|}\right),\\
\nonumber&\lessdot& \kappa_{1} + \sum^{k}_{s=0}\sup_{I,\|I\|\neq0}\left(\frac{(\underset{|j|\leq \mathbb{L}^{s+1}-1}{\sum}\epsilon_{s}^{\frac{8}{15}-\frac{1+\beta}{5}})(\underset{|i|\leq \mathbb{L}^{s+1}-1}{\sum}|I_{i}|e^{|i|^{1+\alpha}})}{\|I\|}\right)\\
\nonumber&\lessdot& \kappa_{1}+\sum^{k}_{s=0}\epsilon_{s}^{\frac{4}{15}}.
\end{eqnarray}
Similarly,\ as
$$ (\bar{\Omega}^{k+1})^{-1}= \begin{pmatrix}
(\Omega^{k+1})^{-1}&0\\
0&0
\end{pmatrix}_{\infty\times\infty},
$$
the operator $ (\bar{\Omega}^{k+1} )^{-1}: \mathbb{C}^{\mathbb{Z}}\rightarrow \mathbb{C}^{\mathbb{Z}} $ satisfies
\begin{eqnarray}
\nonumber|||(\bar{\Omega}^{k+1})^{-1}|||&=& \sup_{I,\|I\|\neq0}\frac{\|(\bar{\Omega}^{k+1})^{-1} I\|}{\|I\|}
=\sup_{\|I\|\neq0}\frac{\|{(\Omega}^{k+1})^{-1} I(k+1)\|}{\|I\|}\\
\nonumber&=& \sup_{I,\|I\|\neq0}\frac{\|(E+\Omega^{-1}(k+1)\sum^{k}_{s=0}\hat{\Omega}^{s})^{-1}\Omega^{-1}(k+1) I(k+1)\|}{\|I\|}\\
\nonumber&\leq& \sup_{I,\|I\|\neq0}\frac{\|(E+\Omega^{-1}(k+1)\sum^{k}_{s=0}\hat{\Omega}^{s})^{-1}b(k+1)\|}{\|I\|}\\
\nonumber&(& \mbox{by letting}\ b(k+1) = \Omega^{-1}(k+1) I(k+1))\\
\nonumber&\leq& \sup_{I,\|I\|\neq0}\frac{\|\sum^{\infty}_{l=0}(\Omega^{-1}(k+1)\sum^{k}_{s=0}\hat{\Omega}^{s})^{l}b(k+1)\|}{\|I\|}\\
\nonumber&\lessdot& \frac{1}{1-{2\kappa_{2}}\sum^{k}_{s=0}\epsilon_{s}^{\frac{4}{15}}}
\left(\sup_{\|I\|\neq0}\frac{\|b(k+1)\|}{\|I\|}\right)\\
\nonumber&\lessdot& \frac{{\kappa_{2}}}{1-{2\kappa_{2}}\sum^{k}_{s=0}\epsilon^{\frac{4}{15}}_{s}}.
\end{eqnarray}

$\mathbf{Step\ 4.}$ $ \mathbf{ Proofs\  of\  (\ref{203})\ and\ (\ref{204})}$.\
From (\ref{050001}),\ (\ref{200}) and accordingly  to the same process of (\ref{026}),\ (\ref{027}),\ we can obtain
\begin{eqnarray}
\label{1026}|P^{low}_{k+1}|_{s_{k}-3\tilde{\sigma}_{k},r_{k}-8\sigma_{k}}&\lessdot & \epsilon_{k}^{1+\beta}\lessdot \epsilon_{k+1},
\end{eqnarray}
and
\begin{eqnarray}
\label{1027}\left|\sum^{4}_{j\geq3} P_{k+1}^{j}\right|_{s_{k}-3\tilde{\sigma}_{k},r_{k}-8\sigma_{k}}&\lessdot& 1+\sum^{k}_{s=0}\epsilon_{k}^{\frac{1+\beta}{2}},
\end{eqnarray}
while we omit the details.
\end{proof}

\section{the measure estimate}
We construct a measure with support at the origin.\ Let
\[
d\sigma(x) = \frac{1}{\sqrt{2\pi}}\exp^{-x^{2}/2}dx
\]
be the standard gaussian measure on the real line with mean zero and variance one,\ and set
\[
d\mu(\omega) = \underset{i\in \mathbb{Z}}{\prod}\frac{1}{\sqrt{2\pi}}\exp^{-\omega_{i}^{2}/2}d\omega_{i}.
\]
Note that
\begin{eqnarray}\label{2031}
\tilde{\mathcal{N}}_{v}^{k} = \left\{ \omega: | \omega(k)\cdot v | \leq (\epsilon^{(1+\beta)^{k-1}})^{\gamma}\right\},
\end{eqnarray}
and
\begin{eqnarray}\label{2032}
\mathcal{N}^{k} = \underset{0< |v| \leq M^{k}}{\bigcup}\tilde{\mathcal{N}}_{v}^{k}
\end{eqnarray}
As in the finite-dimensional case,\ (\ref{2031}) yields
\[
\mu(\tilde{\mathcal{N}}_{v}^{k}) \leq \frac{(\epsilon^{(1+\beta)^{k-1}})^{\gamma}}{||v||_{e}}\leq C^{2\mathbb{L}^{k}}(\epsilon_{k-1})^{\gamma},
\]
where $ ||v||_{e}$ denotes the Euclidean length.\\
Since the number of $ v$ in  (\ref{2032}) is bounded by $ (2M^{k})^{2\mathbb{L}^{k}}$,\ we obtain
\begin{eqnarray}\nonumber
\mu(\mathcal{N}^{k})
\leq C^{2\mathbb{L}^{k}}\epsilon_{k-1}^{\gamma}\cdot(2M^{k})^{2\mathbb{L}^{k}}\leq (\epsilon_{k-1})^{\kappa},
\end{eqnarray}
for some $ 0 < \kappa < \gamma $.\\
Define
\[
\mathcal{R}^{k}=\mathcal{R}^{0}\setminus \underset{1\leq j \leq {k}}{\bigcup}\tilde{\mathcal{N}}^{j}.
\]
Thus, we have
\begin{eqnarray}\label{2033}
\mu(\mathcal{R}^{k})
\leq 1-\overset{k-1}{\underset{j=0}{\sum}}(\epsilon_{j})^{\kappa},
\end{eqnarray}
for some $  0 < \kappa < \gamma  $.

\section{Proof of theorem 1.1}

In this section,\ we prove the  Theorem \ref{0*} by applying Iterative lemma to the Hamiltonian system
defined in (\ref{maineq}).
\begin{proof}
Let $ \mathcal{D}_{*}=\mathcal{D}_{\frac{1}{2}s,\frac{1}{2}r} \subset \bigcap^{\infty}_{0}\mathcal{D}_{s_{k},r_{k}}$,\ $ \mathcal{R}^{\infty}=\underset{j\geq1}{\bigcap}{\mathcal{R}}^{j}$,\ $ \Omega_{*}= \Omega+  \overset{\infty}{\underset{s=0}{\sum}}\hat{\Omega}^{s}$ and $
\Psi_{*} = \prod^{\infty}_{0}\Psi_{k}$.\ By the standard argument,\ we conclude that $ \Psi_{*}, D\Psi_{*}, \Omega_{*}, H_k$ converge uniformly on the domain $\mathcal{D}_{\frac{1}{2}s,\frac{1}{2}r}$.\ Let
\begin{eqnarray}
H_{*} = N_{*} + R_{*},
\end{eqnarray}
where
\begin{eqnarray}
\nonumber N_{*}
&=& e_{*} + \langle \omega,\rho\rangle+ \frac{1}{2}\langle \Omega_{*}\rho,\rho\rangle,
\end{eqnarray}
and
\begin{eqnarray}
\nonumber
{R}_{*} &=&  P_{*}(\rho,\theta),
\end{eqnarray}
with $ P_{*} = \mathcal{O}(|\rho_{i}|^{\iota_{i}}|\rho_{j}|^{\iota_{j}}|\rho_{k}|^{\iota_{k}})$ for $ {\iota_{i}}+{\iota_{j}}+{\iota_{k}}\geq 3$.\\
Moreover,\ by the standard KAM\ proof,\ we obtain the following estimates: \\
(1) The symplectic map $ \Psi_{*} $ satisfies
\begin{eqnarray}
\label{200*} |W_{0}(\Psi_{*}-id)|_{\mathcal{P},\mathcal{D}_{*}}\lessdot \epsilon^{\frac{17}{50}}.
\end{eqnarray}
(2) The operator $ \Omega{*}: \mathbb{C}^{\mathbb{Z}}\rightarrow \mathbb{C}^{\mathbb{Z}} $ satisfies
\begin{eqnarray}
\label{201*}|||\Omega_{*}-\Omega|||\lessdot \sum^{\infty}_{s=0}\epsilon_{s}^{\frac{4}{15}}\lessdot\epsilon^{\frac{1}{5}}.
\end{eqnarray}
(3) The measure of $ \mathcal{R}^{\infty}$ satisfies
\begin{eqnarray}
\label{202*} \mu(\mathcal{R}^{\infty}) \geq 1 - \overset{\infty}{\underset{j=0}{\sum}}(\epsilon_{j})^{\kappa},\
\end{eqnarray}
for some $ 0 < \kappa < \gamma $.
\end{proof}

\footnotesize
\bibliographystyle{abbrv} 

\end{document}